\documentclass[reqno,12pt]{amsart}
\usepackage{amsmath,amssymb,latexsym}	
\numberwithin{table}{section}

\newcommand{\Sp}{\mathrm{Sp}}

\newcommand{\GL}{\mathrm{GL}}
\newcommand{\SL}{\mathrm{SL}}

\newcommand{\R}{\mathbb{R}}
\newcommand{\Q}{\mathbb{Q}}
\newcommand{\Z}{\mathbb{Z}}

\newcommand{\C}{\mathbb{C}}
\newcommand{\e}{\epsilon}

\newtheorem{theorem}{Theorem}[section]
\newtheorem{lemma}[theorem]{Lemma}

\newtheorem{corollary}[theorem]{Corollary}
\theoremstyle{definition}

\newtheorem{remark}[theorem]{Remark}

\author{Sandip  Singh  and  T.  N. Venkataramana}  
\address{School  Of Mathematics Tata Institute of  Fundamental Research, 
Homi Bhabha Road,Mumbai       400005,      India}      
\email{sandips@math.tifr.res.in, venky@math.tifr.res.in}  
\thanks{\small  }  

\begin{document}     \title[Arithmetic    Symplectic    Hypergeometric
Groups]{Arithmeticity of Certain Symplectic Hypergeometric Groups}

\vskip 5mm
\begin{abstract}  We  give  a   sufficient  condition  on  a  pair  of
(primitive)  polynomials  that  the  associated  hypergeometric  group
(monodromy  group  of  the corresponding  hypergeometric  differential
equation) is an arithmetic subgroup of the integral symplectic group.
\end{abstract}
\maketitle

\section{Introduction}

The  main  result  in   this  article  concerns  monodromy  groups  of
hypergeometric  equations of  type $_nF_{n-1}$.   These  equations are
investigated in detail in a paper of Beukers-Heckman (\cite{BH}).  We
recall briefly,  some of  the results and  notation of  \cite{BH}. Set
$\theta  = z\frac{d}{dz}$, viewed  as a  differential operator  on the
Riemann  surface  $S= {\mathbb  P}^1\setminus  \{0,1,\infty \}$.   Let
$\alpha $ be a complex $n$-tuple $(\alpha _1, \ldots, \alpha _n)\in \C
^n$  and  $\beta  $  a  complex  $n$-tuple  $(\beta  _1,  \ldots,\beta
_n)$. Then
\[D=  (\theta  +\beta  _1-1)\cdots  (\theta +\beta  _n-1)-  z  (\theta
+\alpha _1)\cdots (\theta +\alpha _n),\] is a differential operator on
$S= {\mathbb P}^1\setminus \{0,1,\infty\}$.  The differential equation
$Du=0$  on $S$  is called  the  {\it hypergeometric  equation of  type
$_nF_{n-1}$}.  The differential equation  is regular everywhere on $S$
and has regular singularities at $0,1,\infty$.  \\

The topological fundamental group of $S$ (with base point, say, $\frac
12$), acts on the space  of solutions $u$ of the differential equation
$Du=0$, by  analytic continuation along loops in  $S$ corresponding to
elements of the fundamental group.  The fundamental group $\pi$ of $S$
is  generated   by  $h_{\infty},  h_0,  h_1$   where  $h_{\infty}$  is
represented  by a  small loop  going counterclockwise  around $\infty$
exactly once.  Define similarly $h_0$ and $h_1$. Then $h_0h_1h_{\infty
}=1$ and $\pi$  is the free group on  $h_0,h_1,h_{\infty}$ modulo this
relation  $h_0h_1h_{\infty}=1$.  It  is  also the  free  group on  the
generators $h_0$ and $h_{\infty}^{-1}$. \\
  
If  $j\leq n$, let  $a_j=e^{2\pi i  \alpha _j},  \quad b  _j=e^{2\pi i
\beta  _j}$.  Let $f=\prod  _{j=1}^n (X-a_j)$  and $g=  \prod _{j=1}^n
(X-b_j)$.    Write   $f(X)=X^n+A_{n-1}X^{n-1}+\cdots  +A_1X+A_0$   and
$g(X)=X^n+B_{n-1}X^{n-1}+\cdots +B_1X+B_0$ with $A_i,B_j \in \C$. Form
the companion matrices $A$ and $B$ of $f,g$ respectively; that is:
\[A=\begin{pmatrix} 0 & \cdots & 0 & -A_0  \\ 1 & \cdots & 0 & -A_1 \\
\cdots  &   \cdots  &  \cdots   &  \cdots  \\   0  &  \cdots  &   1  &
-A_{n-1}\end{pmatrix}, \quad  B=\begin{pmatrix} 0 & \cdots &  0 & -B_0
\\ 1 & \cdots &  0 & -B_1 \\ \cdots & \cdots & \cdots  & \cdots \\ 0 &
\cdots & 1 & -B_{n-1}\end{pmatrix} .\]

A result  of Levelt (see \cite  {BH}), says that there  exists a basis
$\{u\}$ of solutions of the  differential equation $Du=0$ on the curve
$S$, with  respect to  which the monodromy  action by  the fundamental
group of $S$  is described as follows.  The  action of $h_{\infty}$ is
by the matrix $B^{-1}$,  that of $h_0 $ is by $A$;  then $h_1$ acts by
$C=A^{-1}B$. The group $\Gamma$ generated by the matrices $A,B$ is
called the {\it hypergeometric group} corresponding to $f,g$.  It is
(by  definition)  also  the  monodromy  group  of  the  hypergeometric
equation considered above. \\

In  \cite{BH},  several   questions  about  the  hypergeometric  group
$\Gamma$  are considered  and answered:  when is  $\Gamma$  finite, or
discrete?  What is the Zariski  closure of $\Gamma$ in $GL_n(\C)$?  In
the present  paper, (we use  the results of  \cite{BH} and) we  give a
sufficient condition  for $\Gamma$ to  be an arithmetic group.   If we
take $f,g$ to have  integral coefficients with $f(0)=g(0)=\pm 1$, then
$\Gamma$ lies  in $GL_n(\Z)$  and is hence  discrete.  We  assume this
from now on. We then say that $\Gamma$ is {\it arithmetic}, if $\Gamma
\subset {\mathcal G}(\Z)$ has finite  index, where $\mathcal G$ is the
Zariski closure of $\Gamma$ in $GL_n$  (we remark that this is not the
most general  definition of  an arithmetic group,  but given  that our
group  $\Gamma$  is  a  subgroup  of $GL_n(\Z)$,  this  is  a  natural
definition). \\

The matrices $A$ and $B$  are quasi-unipotent exactly when $\alpha _j$
and $\beta _j$  are rational numbers. That is, the  roots of $f,g$ are
roots  of unity  (this  is  the situation  which  occurs naturally  in
algebraic geometry). In  the following however, we do  not assume that
the roots of  $f,g$ are roots of unity. We will  state our main result
in purely elementary terms, noting  however, that our motivation is to
compute monodromy of hypergeometric equations. Let $f, g \in \Z[X]$ be
monic  polynomials  of degree  $n$  which  are  {\it reciprocal}  i.e.
$X^nf(\frac 1X)=f(X),  \quad X^ng(\frac 1X)=g(X)$  (then $A_0=B_0=1$).
(Sometimes, $f,g$  are called {\it self-reciprocal}).   As before, let
$A,B$ be their companion matrices. \\

We will assume further, that $f,g$  do not have a common root in $\C$.
Assume also that $f,g$ is a  {\it primitive pair}. That is, there does
not   exist  an  integer   $k  \geq   2$  such   that  $f(X)=f_1(X^k),
g(X)=g_1(X^k)$ for  some polynomials  $f_1,g_1\in \Z[X]$.  It  is then
known by a criterion of Beukers and Heckman (\cite{BH}) that $\Gamma $
preserves  a  non-degenerate integral  symplectic  form  $\Omega $  on
$\Z^n$ and that $\Gamma \subset\Sp_{\Omega }$ is Zariski dense. 
In particular, it follows that $n$ is even. \\

The following  problem has been considered by  many people (\cite{S}):
characterise the polynomials $f,g$ for  which the group $\Gamma$ is an
arithmetic group (i.e. of finite  index in $\Sp_{\Omega }(\Z)$). It is
not very easy to produce  Zariski dense subgroups of infinite index in
arithmetic  groups  (such  groups  are  called {\it  thin  groups}  in
\cite{S}), especially  those which are  not free products. One  of the
motivations  for  looking  at   these  hypergeometric  groups  is  the
expectation  that most  of  these  groups should  be  thin.  There  is
evidence for this (\cite{FMS}  , \cite{BT} and \cite{F}); however, the
results of  the present  paper show that  the groups $\Gamma$  are not
{\it always} thin.  \\

In this paper  we give a {\it sufficient}  condition for arithmeticity
of  $\Gamma$.  To  describe  the  condition,  write  $h=f-g$  for  the
difference polynomial. This is a polynomial with integer coefficients:
\[h(X)=cX^d+c_{d-1}X^{d-1}+\cdots + c_rX^r,\] for some $d\leq n-1$ and
some $r$  with $1\leq r \leq  d$, with the leading  coefficient of $h$
denoted by $c$.

\begin{theorem}  \label{mainth} Suppose  $n\geq 4$  {\rm (}is  an even
integer{\rm )}.  Suppose $f,g$  satisfy the foregoing hypotheses: that
is,  $f,g  $  are  monic  polynomials  of  degree  $n$  with  integral
coefficients,  which have  no common  roots, are  {\rm  (}self-{\rm )}
reciprocal, and  form a primitive pair, with  $f(0)=g(0)=1$. Let $A,B$
be   their   respective  companion   matrices   and  $\Gamma   \subset
SL_n({\mathbb Z})$ the subgroup generated by $A$ and $B$.  \\

Assume  in  addition, that  the  leading  coefficient  $c$ of  $h=f-g$
satisfies
\[\mid c  \mid \leq  2.\] Then  the group $\Gamma  $ is  an arithmetic
group. \\

In  particular, if  the difference  $h=f-g$ is  monic, then  the group
$\Gamma $ is arithmetic.
\end{theorem}

Theorem \ref{mainth} holds even when $n=2$, but this is easy (cf. Lemma
\ref{2by2matrix})  and therefore  we do  not describe  the  proof. For
$n\geq  4$,  Theorem  \ref{mainth}  is  proved  by  showing  that  the
``reflection  subgroup''  (\cite{BH})  $\Gamma  _r$ generated  by  the
conjugates $\{ A^{-k}CA^k; \quad k\in \Z\}$ of the element $C=A^{-1}B$
is    arithmetic;    the    inclusion   $\Gamma    _r\subset    \Gamma
\subset\Sp_{\Omega  }(\Z)$ then  shows that  $\Gamma $  is arithmetic.
The element $C$ is a {\it  complex reflection} i.e.  it is identity on
a  codimension one  subspace of  $\Q  ^n$ (it  is also  called a  {\it
transvection},  in   the  symplectic   case).   We  will   deduce  the
arithmeticity  of  the  reflection   subgroup  $\Gamma  _r$  from  the
following  result on subgroups  of $\Sp  _{\Omega} (\Z)$  generated by
certain transvections. 

\begin{theorem} \label{criterion} 

Suppose $\Omega $ is a non-degenerate symplectic form on $\Q ^n$ which
is  integral on  the standard  lattice $\Z^n$.   Suppose  that $\Gamma
\subset\Sp_{\Omega }(\Z)$  is a Zariski dense  subgroup which contains
three  transvections  $C_1,  C_2,  C_3$  such that  if  we  write  $\Z
w_i=(C_i-1)(\Z   ^n)$,  then   $\Omega  (w_i,w_j)\neq   0$   for  some
$i,j$. Assume that $W=\sum _{i=1}^3  \Q w_i$ is three dimensional, and
that  the group  generated by  the restrictions  of the  $C_i$  to $W$
($i=1,2,3$) contains a non-trivial element of the unipotent radical of
the symplectic group of $W$. \\

Under the preceding assumptions, the
group $\Gamma $ has finite index in $\Sp_{\Omega }(\Z)$.
\end{theorem}

In  Section \ref{n=4},  we will  prove the  arithmeticity  of $\Gamma$
directly  (without  assuming   Theorem  \ref{criterion})  for  certain
examples, when $n=4$. These include all the examples of \cite{AvEvSZ},
\cite{DoMo}  (see  also \cite{CEYY})  for  which  we  {\it can}  prove
arithmeticity.  This illustrates the  general method and also gives us
certain non-trivial relations among the generators $A,B$. \\

In  Section \ref{mainthproof},  Theorem \ref{mainth}  is  deduced from
Theorem \ref{criterion}  by checking that three  generic conjugates of
the  transvection  $C=A^{-1}B$   satisfy  the  conditions  of  Theorem
\ref{criterion}, under  the assumption that  the difference polynomial
$f-g$ has  leading coefficient at most  two in absolute  value. In the
course of  the proof we  have to deal  with an associated  subgroup of
$\SL_2(\Z)$ that we get from  the group $\Gamma$, which is generated by
the $2\times 2$ matrices
\[ \begin{pmatrix} 1 & c  \\ 0 & 1\end{pmatrix}, \quad \begin{pmatrix}
1  & 0  \\ c  &  1 \end{pmatrix}.\]  The assumption  that the  leading
coefficient  $c$ is  $\leq 2$  in  absolute value,  implies that  this
subgroup is  of finite index in  $\SL_2(\Z)$. We use this  fact in the
proof  of arithmeticity.  This  is the  reason that  we are  unable to
extend  the proof  to other  $c$'s. Note  that there  are  examples of
groups  of  symplectic  type, when  $\mid  c  \mid  \geq 4$  and  the
monodromy group $\Gamma $ is of infinite index by \cite{BT}. \\

In Section \ref{criterionproof}, Theorem \ref{criterion} will be proved
by showing that with respect to  a suitable basis, the group $\Gamma $
has non-trivial  intersection with the highest and  the second highest
root groups of the  symplectic group $\Sp_{\Omega}$ (and using Theorem
3.5 of \cite{Ve}). \\

In the  last section, we give  tables for pairs $(f,g)$  of degree $4$
monic   polynomials   with  integer   coefficients,   for  which   the
corresponding group $\Gamma $ is  Zariski dense in $\Sp _4$, and which
have  roots of  unity as  roots.  We list  those pairs  for which  the
criterion of Theorem \ref{mainth} applies (and also those for which it
does not).  In turns  out, that  in a (small)  number of  examples the
criterion of  Theorem \ref{mainth}  does not apply,  but arithmeticity
can  still be proved,  by using  Theorem \ref{criterion},  rather than
Theorem  \ref{mainth} directly  (see  Remark \ref{arithmetictable}  of
Section \ref{remarks}). \\

Prescribing a representation  of a free group $F_2$  on two generators
is  the same  as  giving  two matrices  in  $GL_n(\C)$.  Consider  the
representation  of $F_2$  correspnding to  the two  companion matrices
$A,B\in  GL_n(\C)$ as  above. The  group $F_2$  may be  viewed  as the
topological fundamental group  of the curve $S={\mathbb P}^1-\setminus
\{0,1,\infty\}$, generated by a  small loop around $0$ (denoted $h_0$)
and  another (denoted  $h_{\infty}$) around  $\infty$.  The  case when
$n=4$ and  $f=(X-1)^4$ has  generated much interest.   By the  work of
Candelas and  others (\cite{CdOGP}), the  corresponding hypergeometric
group is the  same as the image of  the monodromy representation which
arises from  families of Calabi-Yau threefolds: there  are families of
varieties (Calabi-Yau threefolds) which  fibre over the space $S$, for
which  the  monodromy  action  (on  certain subspaces  of  the  middle
cohomology  of  the fibre  (threefold))  corresponds  to the  matrices
$h_0\mapsto   A,\quad  h_{\infty}   \mapsto   B$  (see   \cite{CdOGP},
\cite{AvEvSZ}, \cite{DoMo}, \cite{CEYY} for details).  \\

The matrix  $A$ is {\it  maximally unipotent} i.e. corresponds  to the
polynomial   $f(X)=(X-1)^4$;   there  are   14   pairs  $(f,g)$   with
$f=(X-1)^4$, $g\in \Z[X]$ is coprime to $f$ and is such that the roots
of  $g$  are  roots  of  unity.  They  are  listed  in  \cite{AvEvSZ},
\cite{DoMo} and \cite{CEYY}.  In one  of the examples of \cite{DoMo} (
Example  13  of \cite{AvEvSZ}  and  Example  13  of \cite{CEYY}),  $f=
(X-1)^4$ and $g=(X^2-X+1)^2$ so that the  roots of $g$ are of the form
$e^{2 \pi i  \beta _j}$ with $\beta _1=\beta  _2= \frac16, \quad \beta
_3=\beta _4=\frac56$.   This pair satisfies the  conditions of Theorem
\ref{mainth} (see the  first example of Table \ref{table:arithmeticity
known}) and hence from Theorem \ref{mainth}, we have the following

\begin{corollary} The  hypergeometric group $\Gamma  \subset Sp_4(\Z)$
corresponding to the polynomials $f=(X-1)^4$ and $g=(X^2-X+1)^2$ is an
arithmetic group.
\end{corollary}

Note that  no other  example in Table  1 of \cite{DoMo}  satisfies the
hypothesis of the Theorem \ref{mainth}. \\

However,   by  using  Theorem   \ref{criterion}  instead   of  Theorem
\ref{mainth}  arithmeticity can  be  proved in  two  more examples  of
\cite{AvEvSZ},  \cite{DoMo} (see  examples  \ref{arithmeticYYCE-2} and
\ref{arithmeticYYCE-3} of Table \ref{table:calabiyau} of
the present paper).

\begin{corollary} The group $\Gamma $ is arithmetic, for the following
pairs (f,g):
\[f= (X-1)^4, \qquad g= (X^2+1)(X^2-X+1), \]  
\[{\rm and} \quad f= (X-1)^4, \qquad g = \frac{X^5+1}{X+1}.\] 
\end{corollary}

In subsections \ref{example2}, \ref{example12} and \ref{example13}, we
prove  arithmeticity of  $\Gamma$ directly  (without appealing  to the
general proof of Theorem \ref{mainth}),  for these 3 examples of Table
1 in \cite{DoMo} considered in the preceding two corollaries.

\begin{remark} \label{CalabiYau}  A recent  result by Brav  and Thomas
(\cite{BT})    says    that    when    $n=4$    and    $f(X)=(X-1)^4$,
$g(X)=\frac{X^5-1}{X-1}$ (in  this case,  all the coefficients  of $h$
are  $\pm 5$),  the group  $\Gamma $  is {\it  thin} in  the  sense of
\cite{FMS}). That is,  $\Gamma $ is of infinite  index in the integral
symplectic group.  In  fact \cite{BT} prove much more:  they show that
in   Examples  \ref{brav1},  \ref{brav2},   \ref{brav3},  \ref{brav4},
\ref{brav5},       \ref{brav6},       \ref{brav7}       of       Table
\ref{table:calabiyau}, the monodromy group is thin.  \\

Therefore, there are  7 thin examples and 3  arithmetic ones among the
14 examples of \cite{AvEvSZ}, \cite{DoMo} and \cite{CEYY}. It would be
interesting to  determine whether the monodromy is  thin or arithmetic
in the remaining 4 examples of the above references. \\

In  the  3  examples  of  $\Gamma$  (see  subsections  \ref{example2},
\ref{example12}  and \ref{example13}  of  this paper)  or  Table 1  of
\cite{DoMo} where  arithmeticity can be  proved, we have  subgroups of
finite  index  in $Sp_4(\Z)$.   In  view  of  the congruence  subgroup
property  for $\Sp_4(\Z)$  \cite{BMS}, it  is clear  that  $\Gamma$ is
(with  respect   to  a  suitable  basis)  a   congruence  subgroup  of
$\Sp_4(\Z)$.  The precise  index  of the  congruence subgroup  $\Gamma
\subset Sp_4(\Z)$   has  been computed by  J.  Hofmann and  Duco van
Straten in \cite{HvS}. This may be useful for arithmetic applications.

\end{remark}

\begin{remark} \label{orthogonal}

An  analogous question  can be  asked when  $f,g$ are  as  before, but
$f(0)=1$  and  $g(0)=-1$. Then,  by  \cite{BH},  the  group $\Gamma  $
preserves  a  nondegenerate {\it  quadratic  form}  $Q$ with  integral
coefficients and the Zariski  closure of $\Gamma $ is $\mathrm{O}(n)$.
Then   the    question   would   be   whether    the   group   $\Gamma
\subset\mathrm{O}_Q  (\Z)$   is  of  finite  index   in  the  integral
orthogonal  group   $\mathrm{O}_Q(\Z)$.   If  the   signature  of  the
quadratic form $Q$  over $\R$ is $(n-1,1)$, then  \cite{FMS} give {\it
seven} infinite  families of  examples when $\Gamma  $ is  of infinite
index (i.e.  is  {\it thin}), producing perhaps the  first examples of
thin monodromy groups with  absolutely simple Zariski closure (namely,
$\mathrm{SO}(n)$).   If the  signature is  $(p,q)$ with  both $p,q>1$,
then we do not know a criterion analogous to Theorem \ref{mainth}.
\end{remark}

\begin{remark} \label{moreexamples}
The  hypergeometric monodromy  group  $\Gamma  $  is relevant  to
algebraic geometry  when $f,g$ are products  of cyclotomic polynomials
(i.e.  when  their roots  are roots of  unity). This is  equivalent to
saying  that  the  ``local   monodromy''  elements  $A$  and  $B$  are
quasi-unipotent elements  (i.e.  some power  of $A$ (resp. of  $B$) is
unipotent).   Theorem \ref{mainth}  says  for example  that if  $f(X)=
\frac{X^5-1}{X-1}$ and $g(X)=(X^2+1)^2$, then  $\Gamma $ is a subgroup
of finite  index in $\Sp_4(\Z)$.  Similar examples can  be constructed
for any even integer $n$. \\

For example, let $m \geq 2$ be an integer and let 
\[f(X)=\frac{X^{2m+1}-1}{X-1}, \quad  g(X)=(X^m-1)^2.\] Then $f,g$ are
coprime, form a primitive pair and the leading coefficient of $f-g$ is
$1$. Hence the associated $\Gamma $ is an arithmetic group. \\

In   \cite{Ve2},  we   prove,  by   a  different   method,   that  the
hypergeometric group  $\Gamma $ associated to the  polynomials ($n$ is
even)
\[f(X)=  \frac{X^{n+1}+1}{X+1}, \quad g(X)=\frac{(X-1)(X^n-1)}{X+1}\]
is  an  arithmetic subgroup  of  $\Sp _n(\Z)$.  This  also follows  from
Theorem \ref{mainth}  since the leading coefficient  of the difference
$f-g$ is $\pm 1$. \\

Our computations show that when $n=4$ and the roots of $f,g$ are roots
of  unity, then  for many  polynomials $f,g$  (60 examples:  see Table
\ref{table:arithmeticity    known})   the   hypotheses    of   Theorem
\ref{mainth} are satisfied and hence the monodromy is arithmetic. \\

However, in Theorem  \ref{mainth}, we do not assume  that the roots of
$f,g$  are  roots  of  unity,  since  the  proof  holds  without  this
assumption. 
\end{remark}

\begin{remark} \label{jannsen}
Theorem \ref{criterion} is  somewhat analogous to the criterion of
Janssen  (\cite{Jan}).  We  have verified  that if  the  assumption of
Janssen's  paper holds  in  our case  (the  assumption on  ``vanishing
lattices''  having  vectors $\delta  _1,  \delta  _2$  with such  that
$\Omega (\delta  _1, \delta _2)=1$  where $\Omega $ is  the symplectic
form given  by \cite{BH}), then  the difference polynomial  $h=f-g$ is
monic.   The proof  of  Theorem \ref{criterion}  shows that  Janssen's
Theorem holds true  if we assume that $\delta  _1.\delta _2=2$ for two
vectors $\delta _1, \delta  _2$ in the vanishing lattice corresponding
to the transvections considered in \cite{Jan}.
\end{remark}

\section*{Acknowledgements}  We  thank   Peter  Sarnak  for
mentioning the problem of  arithmeticity for the hypergeometric groups
(see \cite{S})  and for his  exciting lectures on these  topics during
his visit to Tata Institute. \\

We  also thank  Wadim  Zudilin for  pointing  out a  correction (to  a
relation  which  we  had  wrongly  claimed  to  be  true)  in  Section
\ref{n=4}, and for his help in getting at the correct relation. \\

It is a pleasure to thank  the referees for their very careful reading
of the MS and for many suggestions which have made the paper much more
readable; we also thank the referees for supplying us some important 
references. 

\newpage

\section{Proof  of Arithmeticity  in Illustrative  Special  
Cases when $n=4$} \label{n=4}

In  this section,  we will  prove the  arithmeticity of  the monodromy
group $\Gamma$ directly (without appealing to Theorem \ref{criterion})
for $n=4$, and  for specific examples. The examples  include the pairs
$(f,g)$  such that  $f=(X-1)^4$ (with  $\Gamma$ an  arithmetic group).
This  assumption  on $f$  corresponds  to  the  families considered  in
\cite{AvEvSZ} and Table 1 of \cite{DoMo}) for which we {\it can} prove
arithmeticity of the hypergeometric group.  This will perhaps make the
proof  in the  general  case  more transparent.   We  now explain  the
strategy. \\

Suppose  that  $V$  is  a  four dimensional  $\Q$-vector  space  with a
non-degenerate  symplectic  form   $\Omega$.  First  observe  that  if
$C_1,C_2,C_3$ are  three transvections, then they fix  a non-zero vector
$e$   and  hence  fix   pointwise  the   one  dimensional   space  $\Q
e$. Therefore, the three reflections fix the partial flag
\[\Q e \subset e^{\perp} \subset V.\] The stabiliser of this flag is a
parabolic  $\Q$-subgroup $P$ of  $Sp_{\Omega}$. The  unipotent radical
$U$ of $P$  is the subgroup of $P$ which  acts trivially on successive
quotients of this flag. The quotient $M=P/U$ is isomorphic to $SL_2$. \\

Therefore, the  three transvections  belong to the  parabolic subgroup
$P$.  The image  of the group generated by  the three transvections in
$M=P/U$ lies in the group $M(\Z)=SL_2(\Z)$. Our strategy is to produce
relations among the images of  the three transvections in this copy of
$SL_2(\Z)=M(\Z)$, for  suitably chosen transvections  $C_1,C_2,C_3$ in
the monodromy group $\Gamma$. This  means that there exists a word $w$
in the  $C_i$ which  lies in the  kernel of the  map $P(\Z)\rightarrow
M(\Z)$; that is,  $w$ lies in $U(\Z)$. Once we  get a non-trivial word
$w \in U(\Z)\setminus \{1\}$, by  taking enough conjugates of the word
$w$ in  the group generated  by the $C_i$,  we get the  full unipotent
radical  $U(\Z)$  (or a  subgroup  of  finite  index thereof)  in  the
monodromy group. \\

\subsection{An Example} To illustrate the method, we first take up 
the example
\[f(X)=(X-1)(X^3-1),  \quad  g(X)=X^4+1.\]   By  Beukers  and  Heckman
(\cite{BH}), one  knows that there exists  a non-degenerate symplectic
form  $\Omega$  left invariant  by  $A,B$  where  $A,B$ are  companion
matrices of  $f,g$.  Since $Ae_i=Be_i$ for $i=1,2,3$,  it follows that
if $C=A^{-1}B$, then $C$ fixes $e_1,e_2,e_3$. Since the determinant of
$C$  is one,  $Ce_4=e_4+v$  with some  vector  $v$ which  is a  linear
combination of $e_1,e_2,e_3$; Using the matrices $A,B$, we compute $v=
-e_1-e_3$.   Since $Ae_1=e_2$  and  $Ae_2=e_3$, it  follows that  $C$,
$A^{-1}CA$ and $A^{-2}CA^2$ all fix $e_1$.  \\

We first determine the symplectic form $\Omega$. To ease the notation, 
if $x,y \in  \Q^4$, we write $x.y$ for the number  $\Omega (x,y)$. 
Now, the invariance of  $\Omega $ under
$C$ says that $e_i.e_4= Ce_i.Ce_4= e_i.(e_4+v)$ for $i=1,2,3$.  Hence
$v$ is orthogonal to  $e_1,e_2,e_3$. By non-degeneracy of $\Omega$, we
then  get $v.e_4\neq  0$.  Hence we  have $e_1.e_3=e_1.(e_1+e_3)=0$  and
$e_4.(e_1+e_3)= - e_1.e_4-e_3.e_4\neq 0$. \\

We  now use  the invariance  of $\Omega  $ under  $A$ and  $B$  to get
$e_1.e_2=Ae_1.Ae_2=   e_2.e_3=e_3.e_4$    and   $e_1.e_4=   Be_1.Be_4=
e_2.(-e_1)=e_1.e_2$.   From  the   preceding  paragraph  we  get  that
$e_1.e_4+e_1.e_2\neq 0$.  Hence $e_1.e_2\neq  0$ and once $e_1.e_2$ is
fixed, all the other $e_i.e_j$ are determined. We can then assume that
$e_1.e_2=1$.    We  have  $e_1.e_3=0$   and  $e_1.e_4=1$.   Hence  the
perpendicular of  $e_1$ is the  span of $e_1,e_3,  e_2-e_4$. Moreover,
$e_1.e_4=1$.  \\

The   elements   $C_1=C,   C_2=A^{-1}CA,   C_3=A^{-2}CA^2$   all   fix
$e_1$. Therefore, they preserve the flag
\begin{equation} \label{example} \Q e_1 \subset e_1^{\perp}  \subset V.
\end{equation} The images of $e_1+e_3,e_4-e_2$
in the quotient $e_1^{\perp}/\Q e_1$  form a basis of the quotient. We
have  thus a  basis $e_1,e_1+e_3,e_2-e_4,e_2$  of $V$. We compute  
the matrices of the three reflections with respect to this basis, 
and find that 
\[C_1= \begin{pmatrix} 1 & 0 &  0 & 0 \\ 0 & 1 & 1 & 0 \\  0 & 0 & 1 & 0
\\ 0 & 0 & 0 & 1 \end{pmatrix}, C_2=\begin{pmatrix} 1 & \quad 0 &
0  & 0  \\ 0  &  \quad 0 &  1 &  0  \\ 0  & -1  &  2 &  0 \\  0  & \quad 0  & 0  &
1 \end{pmatrix}, C_3= \begin{pmatrix} 1 & -2 & 0 & -2  \\ 0 
& \quad 1 & 0 & \quad 0 \\ 0 &
-1 & 1 & -1 \\ 0 & \quad 0  & 0 & \quad 1 \end{pmatrix}. \] A computation shows that
$C_2C_1^2= \begin{pmatrix} 1 & \quad 0 & 0 & 0 \\ 0  & \quad 0 & 1 & 0 \\ 0 & -1 & 0 &
0 \\ 0 & \quad 0 & 0 & 1\end{pmatrix}$, and hence that
\[w= (C_2C_1^2)^{-1}C_1(C_2C_1^2)= \begin{pmatrix} 1 & \quad  0 & 0 & 0 \\ 0
& \quad 1 & 0 & 0 \\ 0 & -1 & 1 & 0 \\ 0 & \quad 0 & 0 & 1 \end{pmatrix}. \] Write
$E=w^{-1}C_3$ and $F=C_1^{-1}w^{-1}C_1C_3C_1w^{-1}$. Then
\[E= \begin{pmatrix} 1 & -2 & 0 & -2  \\ 0 & \quad 1 & 0 &\quad 0 \\ 0
&\quad  0 &  1 &  -1\\ 0  &\quad 0  & 0  &\quad  1\end{pmatrix}, \quad
F= \begin{pmatrix} 1 & -4 & -2 & -2 \\ 0 &\quad 1 &\quad 0 &\quad 1 \\
0  &\quad   0  &\quad  1  &  -2   \\  0  &\quad  0   &\quad  0  &\quad
1  \end{pmatrix}.\]  

Some  explanation as  to  the choices  of  $E,F$: $E$  is patently  an
element of the  group of integral points of  the unipotent radical $U$
of the parabolic subgroup $P$  which fixes the line through $e_1$; $E$
lies in the  monodromy group $\Gamma$.  Our strategy is  that if we get
{\it one}  non-trivial element $E$  of $U({\mathbb Z})$ in  $\Gamma$ ,
then  the group  generated by  the conjugates  of $E$  by  elements of
$L\cap  \Gamma$  where $L$  is  the  algebraic  subgroup generated  by
$C_1,C_2$ in  $P$, we get many  others in $U\cap  \Gamma$. The element
$F$ is merely one such matrix with small entries and was arrived at by
inspection.  \\

Given two invertible matrices $a,b$ we write $[a,b]=aba^{-1}b^{-1}$ for 
the commutator of $a$ and $b$. 
Write  $x=[E,F]$,  $y=E^2[E,F]^{-1}$  and
$z=(E^{-2}F)^2[E,F]^3$. Then
\[x=\begin{pmatrix} 1 &0  &0 &-4\\ 0 &1 &0  &\quad 0\\ 0 &0 &1 &\quad 0\\  
0 &0 &0 &\quad 1\end{pmatrix}, y=\begin{pmatrix} 1 &-4 &0 &\quad 0\\  0 
&\quad 1 &0 &\quad 0\\ 0 &\quad 0 &1
&-2\\ 0 &\quad 0 &0 &\quad 1\end{pmatrix},  z=\begin{pmatrix}1 &0 &-4 &0\\ 0 &1 &0
&2\\ 0 &0 &1 &0\\ 0  &0 &0 &1\end{pmatrix}.  \] The elements $C_1,x,y,z$
lie in the unipotent radical  $U$ of the parabolic subgroup determined
by $e_1$: in  other words, they act trivially  on successive quotients
of the flag  of (\ref{example}). It is clear  that the group generated
by  $<C_1,x,y,z>$ is  a finite  index subgroup  of $U(\Z)$,  since these
elements  generate distinct (positive)  root groups.   Hence $C_1,x,y,z$
generate a finite  index subgroup of the group  $U_0(\Z)$ of unipotent
upper  triangular matrices  in the  symplectic group  of  $\Omega$. By
\cite{T},  any  Zariski dense  subgroup  of $Sp_{\Omega}(\Z)$  which
contains a  finite index subgroup  of the group $U_0(\Z)$,  has finite
index in $Sp_{\Omega }(\Z)$.  The  group $\Gamma $ is Zariski dense by
a result of \cite{BH}. Therefore, $\Gamma $ is arithmetic.

\subsection{Example  2  of  \cite{AvEvSZ} and \cite{CEYY}}
\label{example2}  We  
consider Example  \ref{arithmeticYYCE-3}  of  Table  \ref{table:calabiyau}
(this  is Example 2 of \cite{AvEvSZ}). In this case
\[f=(X-1)^4=X^4-4X^3+6X^2-4X+1;  \quad  g=X^4-X^3+X^2-X+1;\] with  the
parameters                                          $\alpha=(0,0,0,0);\
\beta=(\frac{1}{10},\frac{3}{10},\frac{7}{10},\frac{9}{10})$.  Then $A$
and $B$  are given by 
\[A=\begin{pmatrix} 0 &0 &0  &-1\\ 1 &0 &0 &\quad 4\\ 0  &1 &0 &-6\\ 0
&0 &1 &\quad 4 \end{pmatrix}, B=\begin{pmatrix}  0 &0 &0 &-1\\ 1 &0 &0
&\quad 1\\ 0 &1 &0 &-1\\ 0 &0 &1 &\quad 1 \end{pmatrix}\] and let
\[C=A^{-1}B=\begin{pmatrix} 1 &0 &0 &-3\\ 0 &1 &0 &\quad 5\\ 0 &0 &1 &-3\\ 0
&0  &0   &\quad 1  \end{pmatrix}.\]  Let $\Gamma=<A,B>$ be the  subgroup  of
$\SL_4(\Z)$ generated by $A$ and $B$. \\

Let  $\Omega$ be  the symplectic  form (determined up to scalar multiples) 
preserved  by $\Gamma$  and let
$\Omega=(\Omega(e_i,e_j))$   with  respect   to  the   standard  basis
$\{e_1,e_2,e_3,e_4\}$. Then, by considerations  similar to those in the
preceding subsection, the matrix of $\Omega$ may be computed to be (up
to scalar multiples)

\[\Omega=\begin{pmatrix} 0 &1 &5/3 &5/3\\  -1 &0 &1 &5/3\\ -5/3 &-1 &0
&1\\ -5/3 &-5/3 &-1 &0  \end{pmatrix}.\] By an easy computation we get
$\epsilon_1=e_1-e_2+e_3$,                 $\epsilon_2=-3e_1+5e_2-3e_3$,
$\epsilon_2^*=2e_1-5e_2+5e_3-2e_4$ and $\epsilon_1^*=e_1$ form a basis
of $\Q^4$ over $\Q$ with respect to  which
\[\Omega=\begin{pmatrix}
0 &0 &0 &-2/3\\
0 &0 &-2/3 &0\\
0 &2/3 &0 &0\\
2/3 &0 &0 &0 \end{pmatrix}.\]

Let $C_1=C=A^{-1}B,  C_2=B^{-2}CB^2$ and $C_3=B^2CB^{-2}$.   It can be
checked     easily    that     with    respect     to     the    basis
$\{\epsilon_1,\epsilon_2,\epsilon_2^*,\epsilon_1^*\}$,  the $C_i$ have, 
respectively, the matrix form
\[\label{C} \begin{pmatrix} 1 &0 &0 &0\\ 0  &1 &-2 &0\\ 0 &0 &1 &0\\ 0
&0 &0  &1 \end{pmatrix}, \quad  \begin{pmatrix} 1 &0  &0 &0\\ 0  &1 &0
&0\\   0   &2   &1   &0\\   0  &0   &0   &1   \end{pmatrix}   \rm{and}
\quad \begin{pmatrix}  1 &-1 &3/2 &-1/2\\  0 &4 &-9/2 &3/2\\  0 &2 &-2
&1\\  0  &0  &0  &1  \end{pmatrix}.\]  Now,  $C_1,  C_2$  generate  an
arithmetic subgroup  of the Levi  $L=SL_2$ (up to  $\pm 1$, it  is the
principal  congruence subgroup  of  level $2$);  $C_3$  is a  rational
unipotent element of $L$; hence some  power $(C_3)^m$ is a word $w$ in
$C_1,C_2$ and hence, as explained in the beginning of this section, we
get  an  element  $E_0\neq  1$   of  the  unipotent  radical  of  $P$:
$E_0=(C_3)^m w^{-1}$.  By manipulating with suitable matrices, we then
see that if $E=(C_2^2C_1C_3^2C_1^{-1})^2C_1 , F=C_2EC_2^{-1}$ then

\[E=\begin{pmatrix} 1 &0 &2 &2\\ 0 &1 &0 &2\\ 0 &0 &1 &0\\ 0 &0 &0 &1
\end{pmatrix},\quad F=\begin{pmatrix}1 &-4 &2 &2 \\ 0 &1 &0 &2\\
0 &0 &1 &4\\0 &0 &0 &1 \end{pmatrix}.\] Then  
\[x=[E,F]=\begin{pmatrix} 1 &0 &0 &16\\ 0 &1 &0 &0\\ 0 &0 &1 &0\\ 0 &0
&0  &1 \end{pmatrix},\quad  y=E^8[E,F]^{-1}=\begin{pmatrix}  1 &0  &16
&0\\ 0 &1 &0 &16\\ 0 &0 &1 &0\\ 0 &0 &0 &1 \end{pmatrix}. \] Moreover,
\[ z=(FE^{-1})^2[E,F]^{-1}=\begin{pmatrix} 1 &-8  &0 &0\\ 0 &1 &0 &0\\
0 &0 &1 &8\\ 0 &0 &0  &1 \end{pmatrix}. \] The elements $<C, x, y, z>$
patently generate  the positive root groups of  $\Sp_4$.  Therefore if
$\mathrm{B}$ is the Borel  subgroup of $\Sp_\Omega(\Q)$ preserving the
full flag
\[\{0\}\subset\Q\e_1\subset\Q\epsilon_1\oplus\Q\epsilon_2
\subset\Q\epsilon_1
\oplus\Q\epsilon_2\oplus\Q\epsilon_2^*\subset\Q\epsilon_1
\oplus\Q\epsilon_2\oplus\Q\epsilon_2^*\oplus\Q\epsilon_1^*=\Q^4\]   and
$\mathrm{U}$ its unipotent radical, then $\Gamma\cap\mathrm{U}(\Z)$ is
of finite  index in $\mathrm{U}(\Z)$.   Now by \cite{BH},  $\Gamma$ is
Zariski dense in  $\Sp_\Omega$. By replacing $\Gamma $  by a conjugate
by  an  element  $g\in  \Sp_\Omega(\Q)$  if  necessary,  we  see  that
$\Gamma\cap\mathrm{U}^{-}(\Z)$     is    of     finite     index    in
$\mathrm{U}^{-}(\Z)$, where $\mathrm{U}^{-}$  is the unipotent radical
of the  opposite Borel $\mathrm{B}^{-}$ (which  consists of transposes
of elements  of $\mathrm{B}$).   Therefore it follows  from \cite{BMS}
that $\Gamma$ is an arithmetic subgroup of $\Sp_\Omega(\Z)$. Note that
we have an explicit relation
\[[E,[E,F]]=1,\] between  the elements $E,F$  of $\Gamma $  (the group
generated by the matrices $A$ and $B$), where $E,F$ are explicit words in $A,B$.

\subsection{Example  12 of  \cite{AvEvSZ}and \cite{CEYY}}  \label{example12} Consider
Example \ref{arithmeticYYCE-2}  of Table  \ref{table:calabiyau} (this  is
Example        12       of       \cite{AvEvSZ}).         In       this
case  \[f=(X-1)^4=X^4-4X^3+6X^2-4X+1;\quad  g=X^4-X^3+2X^2-X+1\]  with
the                   parameters                   $\alpha=(0,0,0,0);\
\beta=(\frac{1}{4},\frac{3}{4},\frac{1}{6},\frac{5}{6})$. Then
\[A=\begin{pmatrix}
0 &0 &0 &-1\\
1 &0 &0 &\quad 4\\
0 &1 &0 &-6\\
0 &0 &1 &\quad 4 \end{pmatrix}, B=\begin{pmatrix}
0 &0 &0 &-1\\
1 &0 &0 &\quad 1\\
0 &1 &0 &-2\\
0 &0 &1 &\quad 1 \end{pmatrix}, C=A^{-1}B=\begin{pmatrix}
1 &0 &0 &-3\\
0 &1 &0 &\quad 4\\
0 &0 &1 &-3\\
0 &0 &0 &\quad 1 \end{pmatrix}.\]
Let $\Gamma=<A,B>$ be the subgroup of $\SL_4(\Z)$ generated by $A$ and $B$.

As before  the matrix  form of $\Omega$  with respect to  the standard
basis $\{e_1,e_2,e_3,e_4\}$, may be computed to be
\[\begin{pmatrix} 0  &1 &4/3 &1/3\\ -1  &0 &1 &4/3\\ -4/3  &-1 &0 &1\\
-1/3  &-4/3 &-1  &0 \end{pmatrix}.\]  By  an easy  computation we  get:
$\epsilon_1=e_1-e_2+e_3$,                 $\epsilon_2=-3e_1+4e_2-3e_3$,
$\epsilon_2^*=4e_1-5e_2+4e_3-e_4$ and  $\epsilon_1^*=e_1$ form a basis
of $\Q^4$ over $\Q$ with respect to which
\[\Omega=\begin{pmatrix}
0 &0 &0 &-1/3\\
0 &0 &-4/3 &0\\
0 &4/3 &0 &0\\
1/3 &0 &0 &0 \end{pmatrix}.\]
 
Let  $C_1=C=A^{-1}B, C_2=B^{-2}CB^2$ and  $C_3=B^2CB^{-2}$. It  can be
checked     easily    that     with    respect     to     the    basis
$\{\epsilon_1,\epsilon_2,\epsilon_2^*,\epsilon_1^*\}$,

\[ C_1=\begin{pmatrix} 1 &0 &\quad 0 &0\\ 0 &1 &-1 &0\\ 0 &0 &\quad 1 &0\\ 0 &0 &\quad 0
&1 \end{pmatrix},\quad C_2=\begin{pmatrix} 1 &0 &0 &0\\ 0 &1 &0 &0\\ 0
&1 &1 &0\\ 0 &0 &0 &1 \end{pmatrix}, \quad C_3=\begin{pmatrix} 1
&-4 &\quad 4 &-4\\ 0 &\quad 2 &-1 &\quad 1\\ 0  &\quad 1 &\quad 0 &\quad 1\\ 
0 &\quad 0 &\quad 0 &\quad 1 \end{pmatrix}.\] A
computation    shows    that    if   $E=C_2^{-1}C_3C_2C_1^{-1}$    and
$F=C_2EC_2^{-1}$, then
\[E=\begin{pmatrix} 1 &0 &4 &-4\\ 0 &1 &0 &1\\ 0 &0 &1 &0\\ 0 &0 &0 &1
\end{pmatrix},\mbox{ and } F=\begin{pmatrix}
1 &-4 &4 &-4\\
0 &1 &0 &1\\
0 &0 &1 &1\\
0 &0 &0 &1
\end{pmatrix}.\] We have 
\[x=[E,F]=\begin{pmatrix}  1 &0 &0  &8\\ 0
&1 &0 &0\\ 0 &0 &1 &0\\ 0 &0 &0 &1
\end{pmatrix}, \quad y=E^2[E,F]=\begin{pmatrix}
1 &0 &8 &0\\
0 &1 &0 &2\\
0 &0 &1 &0\\
0 &0 &0 &1
\end{pmatrix}\]
and
\[z=(FE^{-1})^2[E,F]^{-1}=\begin{pmatrix} 1  &-8 &0 &0\\ 0
&1 &0 &0\\ 0 &0 &1 &2\\ 0 &0 &0 &1
\end{pmatrix}. \]
The  elements $<C_1,x,y,z>$ obviously generate  the positive  root  groups  of
$\Sp_\Omega$. Therefore, as before,  $\Gamma $ intersects the group $U(\Z)$ of 
unipotent   upper   triangular   integral   symplectic   matrices   in
$\Sp_{\Omega}(\Z)$ in  a finite  index subgroup. Hence  by \cite{BMS} and by \cite{T},
$\Gamma $ is  arithmetic. The proof also shows  the following relation
between $A,B$:
\[[E,[E,F]]=1\] where $E$ and $F$ are explicit words in $A,B$.

\subsection{Example  13 of  \cite{AvEvSZ} and \cite{CEYY}}  \label{example13} In  this
subsection, we  give the  proof of arithmeticity  of the  monodromy of
Example 13 of  \cite{AvEvSZ} (this is Example \ref{arithmeticYYCE-1} in Table
\ref{table:calabiyau} of the present paper). \\

In       this       case      $f=(X-1)^4=X^4-4X^3+6X^2-4X+1$       and
$g=(X^2-X+1)^2=X^4-2X^3+3X^2-2X+1$       with      the      parameters
$\alpha=(0,0,0,0);\
\beta=(\frac{1}{6},\frac{1}{6},\frac{5}{6},\frac{5}{6})$.   Since  the
difference polynomial  $h=f-g =-2X^3+3X^2-2X$ has  leading coefficient
-2, the arithmeticity  of $\Gamma$ follows directly from
Theorem \ref{mainth}. Write
\[C_1=C=A^{-1}B,   \quad   C_2=ACA^{-1},\quad
C_3=A^{-1}CA,\]       and       \[E=(C_2^2C_1^{-1}C_3^2C_1)^2C_1;\quad
F=C_2EC_2^{-1}.\] Then we have the relation \[[E,[E,F]]=1.\]

\section{Proof of Theorem \ref{criterion}} \label{criterionproof}

\subsection{Subgroups of a semi-direct product}

Consider the natural action of the integral linear group $\SL_2(\Z)$ on
$\Z  ^2$.  Form the  semi-direct product  $\Z^2 \rtimes \SL_2(\Z)$ and
suppose $\Delta  \subset \Z^2 \rtimes\SL_2(\Z)$ is a  subgroup, whose
projection to  $\SL_2(\Z)$ is Zariski dense in  $\SL_2$, and suppose
that $\Delta $ contains a non-trivial element of $\Z^2$.

\begin{lemma}  \label{semidirect} The intersection  of $\Delta  $ with
the integral unipotent subgroup $\Z^2$ has finite index in $\Z ^2$.
\end{lemma}

\begin{proof} Since  $\Z ^2$  is abelian, the  action of $\Delta  $ on
$\Z^2$ factors  through its  projection to $\SL_2(\Z)$.  By assumption,
$\Delta  $  has  Zariski  dense  image  in  $\SL_2$,  and  hence  acts
irreducibly  on $\Q^2$.  Since $\Delta  \cap  \Z ^2$  is non-zero,  it
follows that the  normal subgroup generated by $\Delta  \cap \Z ^2$ in
$\Delta $ contains two  $\Q$-linearly independent elements in $\Z ^2$,
i.e. $\Delta \cap \Z ^2$ has finite index in $\Z ^2$.
\end{proof}

We will  apply this lemma  to the following  situation.  Let $W$  be a
three dimensional  $\Q$-vector space with a  {\it non-zero} symplectic
form $\Omega  $ on  $W$. Since  $W$ is odd  dimensional, $\Omega  $ is
degenerate and hence $W$ has a  null subspace, which must be one
dimensional: $E=\Q   e$  for some $e\in W\setminus \{0\}$. \\

The symplectic group $\Sp_W$ of $\Omega  $ is not reductive, and is in
fact  a semi-direct  product:  $\Sp _W(\Q)=  \Q ^2\rtimes  \SL_2(\Q)$,
where $\Q ^2  = Hom (W/E,E)$ is identified  with the unipotent radical
$\mathrm{U}$   of  $\Sp_W$   by   sending  a   linear  form   $\lambda
\in\mbox{Hom}   (W/E,E)$   to   the  (symplectic)   unipotent   linear
transformation $x\mapsto  x+\lambda (x)$, for all  $x\in W$. Moreover,
$\SL_2(\Q)\simeq  \Sp   _{W/E}$  is   the  symplectic  group   of  the
non-degenerate form defined by $\Omega$ on the quotient $W/E$.  \\

Suppose that $W_{\Z}$ denotes the integral span of a basis of $W$, and
suppose  that $w_1,w_2,w_3\in  W_{\Z}$ are  linearly  independent over
$\Q$.  Denote by $C'_i$ the transvection
\[x\mapsto  x-\Omega (x,w_i)w_i\qquad \forall  x \in  W. \]  Denote by
$\Delta '$ the group generated by the three transvections $C'_1,C'_2$, $C'_3\in
\Sp_{W}(\Z)$.

\begin{lemma}  \label{threedimensional}  If   $\Delta  '$  contains  a
non-trivial element  of $\mathrm{U}(\Z)$,  then $\Delta '$  contains a
subgroup of finite index in $\mathrm{U}(\Z)$.
\end{lemma}

\begin{proof} We will apply  Lemma \ref{semidirect}. Since $\Omega$ is
not the zero  symplectic form on $W$, and $w_1,w_2,w_3$  is a basis of
$W$, it follows that $\Omega (w_i,w_j)\neq  0$ for some $i \neq j$; we
may assume, after a  renumbering, that $\Omega (w_1,w_2)\neq 0$. Hence
the span of $w_1,w_2$ does not  intersect the null space $E=\Q e$, and
$e,w_1,w_2$ form  a basis of  $W$.  We assume,  as we may,  that $e\in
W_{\Z}$. Replacing  $\Delta '$ by a  subgroup of finite  index, we may
assume that $\Delta  '$ preserves the lattice in  $W_{\Q}$ spanned by
$e, w_1,w_2$.  Since the  unipotent radical $\mathrm{U}$ consists only
of (unipotent and hence) elements of infinite order, the assumptions of
the lemma are  not altered if we replace $\Delta  '$ by a finite-index
subgroup. \\

With  respect  to  this  basis  $e,w_1,w_2$, the  group  generated  by
$C'_1,C'_2$ fixes  $e$ and takes $w_1,w_2$ into  linear combination of
$w_1,w_2$. Write $\lambda =\Omega (w_1,w_2)\neq 0$. Then, the matrices
of $C'_1, C'_2$ are
\[C'_1=\begin{pmatrix}  1 &  0  & 0  \\  0 &  1  & \lambda  \\  0 &  0
&1 \end{pmatrix}, \qquad C'_2=\begin{pmatrix}  1 & 0 & 0 \\ 0  & 1 & 0
\\0 & -\lambda & 1 \end{pmatrix}.  \] It is therefore clear that the
group generated  by $C'_1, C'_2$ is Zariski  dense in $\SL_2 \simeq
\Sp_W/\mathrm{U}$.  \\ 

Hence the  assumptions of Lemma \ref{semidirect} are  satisfied and so
is the conclusion.
\end{proof}

\subsection{A  Unipotent Subgroup of  $\Sp_4$} Suppose  that $X$  is a
four  dimensional $\Q$-vector space  with a  non-degenerate symplectic
form  $\Omega _X$.   Suppose  $C_1,C_2, C_3$  are three  transvections
corresponding to vectors $w_1,w_2,w_3$ which are linearly independent,
and such that  $\Omega _X(w_1,w_2)\neq 0$.  If $W$ is  the span of the
$w_i$  and $C_i'$ denotes  the restriction  of $C_i$  to $W$,  then we
assume  that  the   hypotheses  of  Lemma  \ref{threedimensional}  are
satisfied. \\

As before, let  $\mathrm{U}$ be the unipotent radical  of $\Sp_W$, and
$\Q e$  be the null space of  $\Omega $ restricted to  $W$.  Denote by
$\mathrm{P}_X$  the  subgroup  of  the $4\times  4$  symplectic  group
$\Sp_X$,  which preserves  the partial  flag $\Q  e \subset  W \subset
X$.  Then $\mathrm{P}_X$  is  a parabolic  subgroup  of $\Sp_X$.   The
subgroup  of   $\mathrm{P}_X$  which  acts   trivially  on  successive
quotients  of this flag  is precisely  its unipotent  radical, denoted
$\mathrm{U}_X$.  We have the restriction map $\mathrm{P}_X \rightarrow
\Sp_W$, which is easily seen to be surjective. \\

The group  generated by $C_1,C_2,  C_3$ lies in  $\mathrm{P}_X$: since
$W$ is the span of  the images $(C_i-1)$, and $C_i$ are transvections,
it follows that $W$ is stable under the $C_i$'s.  \\

It  is also  trivial to  see that  $\mathrm{U}_X$ is  the  preimage of
$\mathrm{U}$  under  the  restriction  map  $\mathrm{P}_X  \rightarrow
\Sp_W$  and  that  the  kernel  of the  surjective  map  $\mathrm{U}_X
\rightarrow     \mathrm{U}$     is     the     commutator     subgroup
$[\mathrm{U}_X,\mathrm{U}_X]$ (which is one dimensional). \\

If $e\in W$ generates the null space of $W$, the non-degeneracy of the
symplectic  space $X$  shows the  existence  of $e^*\in  X$ such  that
$\Omega    (e,e^*)\neq   0$.    We   assume,    as   we    may,   that
$\Omega(e^*,w_1)=\Omega (e^*,w_2)=0$. Denote  by $X_{\Z}$ the integral
span of the  vectors $e, w_1,w_2, e^*$. Thus  $\Sp_X(\Z)$ is the space
of  symplectic transformations  on $X$  which preserves  this integral
span $X_{\Z}$.  If we choose a different lattice in $X_{\Q}$, then the
resultant integral symplectic group is commensurable with $\Sp_X(\Z)$.

With this notation, we have the following Lemma.

\begin{lemma} \label{fourdimensional} The group $\Delta $ generated by
$C_1,C_2,C_3$ contains a finite-index subgroup of $\mathrm{U}_X(\Z)$.
\end{lemma}

\begin{proof} The  group $\Delta $ maps  onto the group  $\Delta '$ of
Lemma \ref{threedimensional}. By Lemma \ref{threedimensional}, $\Delta
'$    contains   a    finite-index    subgroup   $\mathrm{U}_1'$    of
$\mathrm{U}(\Z)$. The preimage of $\mathrm{U}_1'$ in $\mathrm{P}_X$ is
generated   by   the   kernel  $[\mathrm{U}_X,\mathrm{U}_X]$   and   a
finite-index subgroup  of $  \mathrm{U}_X(\Z)$. Hence the  preimage of
$\mathrm{U}_1'$  lies  in  $U_X$.   Therefore, $\Delta  \cap  U_X$  is
non-trivial;  moreover,  $\Delta  \cap  \mathrm{U}_X$  maps  onto  the
finite-index  subgroup  ${\mathrm U}_1'$  of  $\mathrm{U}(\Z)$.  If  a
subgroup  $H$  of $U_X(\Z)$  maps  onto  a  finite index  subgroup  of
$U(\Z)=$ the abelianisation of $U_X(\Z)$, then it is easily shown that
$H$  has   finite  index   in  $U_X(\Z)$.   Therefore,   $\Delta  \cap
\mathrm{U}_X$ has finite index in $\mathrm{U}_X(\Z)$.
\end{proof}

\begin{corollary} \label{rootgroups} In particular, $\Delta $ contains
the subgroup of matrices of the form
\[ \begin{pmatrix}  1 &  0 &  y_2 &  z \\ 0  & 1  & 0  & \frac{\lambda
_1}{\lambda_2} y_2 \\0  & 0 & 1 &  0 \\ 0 & 0 &  0 & 1\end{pmatrix},\]
with $y_2,z$ in a subgroup of finite index in $\Z$.

\end{corollary}

In     the    above     corollary     $\lambda_1=\Omega(e,e^*)$    and
$\lambda_2=\Omega(w_1,w_2)$.

\subsection{Proof of Theorem \ref{criterion}}

We  will now  prove Theorem  \ref{criterion}. Let  $C_1, C_2,  C_3$ be
three   transvections    satisfying   the   conditions    of   Theorem
\ref{criterion}. They correspond  to linearly independent vectors $w_i
\in \Z^n$. Set $W$ to be their $\Q$-span, and let, as before, $e$ be a
generator of  the null  space in $W$  of the symplectic  form $\Omega$
restricted to $W$. The non-degeneracy  of $\Z^n$ as a symplectic space
implies the  existence of  a vector $e^*\in  \Z ^n$ such  that $\Omega
(e,e^*)\neq 0$.   Let $X$  be the span  of $W$  and $\Q e^*$.   We may
write the orthogonal decomposition
\[\Z^n=X\oplus  X^{\perp}.\]  Hence   the  ``reflections''  $C_i$  act
trivially  on $X^{\perp}$.   Write $\e_1=e$,  $\e_2=w_1,  \e_2 ^*=w_2,
\e_1 ^*=e^*$.  Now a ``symplectic'' basis $\e_3,  \cdots \e_n; \e_n^*,
\cdots,  \e_3^*$   of  $X^{\perp}$  may  be  chosen   so  that  $\e_i$
(resp.   $\e_i^*$)   is   orthogonal   to  all   the   $\epsilon   _j$
(resp.  $\e_j^*$)  and  $\Omega  (\epsilon _i,  \epsilon  _j^*)=\delta
_{ij}$, where $\delta _{ij}$ is $1$ if $i=j$ and $0$ otherwise. \\

Consider the ordered basis
\[\epsilon  _1,  \epsilon  _2,  \cdots, \epsilon_n;  \epsilon  _n  ^*,
\cdots, \epsilon _2^*, \epsilon _1^*,\] of $\Q^n$, and with respect to
this basis, define  the standard Borel subgroup of  $\Sp_n(\Z)$ as the
group of upper triangular matrices in $\Sp_n$, and a maximal torus $T$
to  be the group  of diagonals  in $\Sp_n$.   This datum  determines a
positive  system  of  roots  $\Phi  ^+$  in  the  character  group  of
$\mathrm{T}$.  We can then talk of highest and second highest roots in
$\Phi ^+$.  \\

Consider  the unipotent  upper triangular  matrices  $\mathrm{U}_n$ in
$\Sp_n$  with   respect  to  the   ordered  basis  of   the  preceding
paragraph.  The subgroup  of  $\mathrm{U}_n$ which  acts trivially  on
$\epsilon _j, \epsilon_j^*$ if $j \neq 1,2$, is the group generated by
the highest  and a  (actually there is  only one) second  highest root
groups in $\mathrm{U}_n$. \\

It is  then immediate from  Corollary \ref{rootgroups} that  the group
generated  by $C_1,  C_2, C_3$  intersects  the highest  and a  second
highest  root  groups  of  $\Sp_n(\Z)$ non-trivially.  By  assumption,
$\Gamma $ is a Zariski  dense subgroup containing $C_1, C_2, C_3$.  By
Theorem (3.5) of \cite{Ve}, it follows that $\Gamma $ has finite index
in $\Sp_n(\Z)$.

\section{Proof of Theorem \ref{mainth}} \label{mainthproof}

\subsection{Some Generalities}

Recall that  the group $\Gamma  $ generated by the  companion matrices
$A,B$.  Put  $C=A^{-1}B$. Let $e_1,e_2,  \cdots, e_n$ be  the standard
basis  of  $\Z^n$.   Clearly,  $(C-1)(\Z  ^n)=\Z v$  for  some  vector
$v$. Since $C$  is identity on $e_1, \cdots,  e_{n-1}$, this means that
$(C-1)e_n=\lambda  v$ is a  non-zero multiple  of $v$.   By \cite{BH},
$\Gamma$  preserves  a non-degenerate  symplectic  form  $\Omega $  on
$\Z^n$.   Given  $x,y\in\Z^n$,  write  $x.y$ for  the  number  $\Omega
(x,y)$.

\begin{lemma} \label{vectorv} The vector $v$ is orthogonal {\rm (}with
respect  to the  symplectic  form $\Omega${\rm)}  to  all the  vectors
$e_1,e_2, \cdots, e_{n-1}$.  Moreover, $v$  is a cyclic vector for the
action of $A$ {\rm (}and also for the action of $B${\rm )} on $\Q^n$.
\end{lemma}

\begin{proof}    If     $i\leq    n-1$,    then    $e_i.e_n=Ce_i.Ce_n=
e_i.(e_n+\lambda v)=e_i.e_n+\lambda  e_i.v$. Cancelling $e_i.e_n$ on
the  left and  the  right extreme  sides  of these  equations, we  get
$v.e_i=0$ for $i\leq n-1$. \\

If $v$ is not cyclic, then there exists a polynomial $\phi $ in $X$ of
degree  strictly less than  $n$ such  that $\phi  (A)v=0$. If  we view
$\Q^n  $  as  the  quotient  ring $\Q[X]/(f(X))$,  then  $A$  acts  as
multiplication by $X$,  and $v=\frac{h}{X}$.  Therefore, $\phi (A)v=0$
implies that $\phi (X)\frac{h}{X}$ is divisible by $f(X)$. Since $f,g$
are coprime,  then so  are $h,f$ and  hence the divisibility  of $\phi
\frac hX$  by $f$ implies  that $f$ divides $\phi$,  contradicting the
assumption that $\phi $ has degree less than $n$.

\end{proof}

Since  $A$   and  $B$  are  companion   matrices  with  characteristic
polynomials $f,g$, it follows after an easy computation, that
\[C=\begin{pmatrix} 1 & 0 &  0 & \cdots & c_1 \\ 0 & 1  & 0 & \cdots &
c_2\\ \cdots & \cdots & \cdots & \cdots & \cdots \\ 0 & 0 & \cdots & 1
& c_{n-1} \\  0 & 0 &  0 & \cdots & 1  \end{pmatrix} = \begin{pmatrix}
1_{n-1} & v \\ 0 & 1 \end{pmatrix}, \] where $1_{n-1}$ is the identity
$(n-1)\times (n-1)$-matrix,  $c_i=A_i-B_i$ for $1\leq i  \leq n-1$ and
$v$ (resp.  $0$)  is viewed as a column (resp.   row) vector of length
$n-1$ with entries $c_1,c_2,\cdots,c_{n-1}$ (resp. $0,\cdots,0$). \\

Let $h=f-g$. Then there exists a unique integer $k\leq n/2$ such that 
\[h=      f-g     =     cX^{n-k}+      c_{n-k-1}X^{n-k-1}+\cdots     +
c_{k+1}X^{k+1}+c_kX^k,\]  with $c\neq 0$.  Since $f,g$  are reciprocal
polynomials,  so is  $h$  and $c_k=c$.   The  vector $v$  is a  linear
combination of $e_1, \cdots,e_{n-1}$:
\[v=\sum _{j=k}^{n-k}c_je_j. \quad {\rm Then}, \quad A^k(v)=ce_n+  \sum  _{j=2k}^{n-1}   c_{j-k}e_j.  \]  Moreover,  it
follows from Lemma \ref{vectorv}, that
\[\Omega(v,A^{-k}v)=\Omega (A^kv,v)= c\Omega (e_n,v)\neq 0.\]

\begin{lemma}  \label{independence}  Consider the  two  sets of  three
vectors $S_A=\{v,A^kv,A^{-k}v  \}$ and $S_B=\{v,B^kv,B^{-k}v\}$. Then,
the vectors in either $S_A$ or in $S_B$ are linearly independent.
\end{lemma}

\begin{proof} The vectors $v,A^kv$ are linearly independent, since any
linear  dependence implies  that a  polynomial  in $A$  of degree  $k$
annihilates $v$, contradicting Lemma \ref{vectorv}. \\

Suppose   that  $v,A^kv,A^{-k}v$  are   linearly  dependent.   Then  there 
exists a polynomial of  the form $p(X^k)$ with  $p$ of degree two  such that
$p(A^k)v=0$. By Lemma \ref{vectorv}, $v$  is cyclic and $2k\leq n$ and
hence $f(X)=p(X^k)$ with $k=\frac{n}{2}$. \\

Similarly,  we  have  a  polynomial   $q$  of  degree  two  such  that
$g(X)=q(X^k)$ if $v,B^kv,B^{-k}v$ are linearly dependent. \\

The conclusion  of the last two paragraphs  contradicts the assumption
that $f,g$ form a primitive pair. Therefore, the lemma follows.
\end{proof}

\begin{remark} The lemma holds under the weaker assumption that $f,g$ 
are not polynomials in $X^k$ with $2k=n$.  However, in the course of the 
proof of Theorem \ref{mainth}, we repeatedly use the result of \cite{BH} 
that $\Gamma$ is Zariski dense in $Sp_n$; and that result requires that 
$f,g$ are a primitive pair.   
\end{remark}

We  assume  henceforth,  that  $w_1=v,  w_2=A^{-k}v,  w_3=A^{k}v$  are
linearly independent. Let $W$  denote the $\Q$-vector space spanned by
the $\{w_i\}$ and let $W_{\Z}$  denote the integral linear span of the
$w_i$.  Consider the  transvections  $C_1=C=A^{-1}B$, $C_2=A^{-k}CA^k$
and   $C_3=A^kCA^{-k}$.  The   images  of   $C_i-1$  are   spanned  by
$w_i$ ($i=1,2,3$).

\begin{lemma} \label{2by2matrix} The $\Z$  span of $w_1,w_2$ is stable
under the action  of $C_1,C_2$. Moreover, with respect  to this basis,
$C_1, C_2$ have {\rm(}respectively{\rm)} the matrix form
\[M_1=   \begin{pmatrix}  1   &-c   \\  0   & \quad 1\end{pmatrix},   \quad
M_2= \begin{pmatrix}  1 & 0 \\  c & 1\end{pmatrix}.\] If  $\mid c \mid
\leq 2$, then as linear transformations on the $\Q$-span of $w_1,w_2$,
the  group  generated by  $M_1,  M_2$  is  an arithmetic  subgroup  of
$\SL_2(\Z)$.
\end{lemma}

\begin{proof} Recall  that $A^{-k}v=-ce_n+v'$  where $v'$ is  a linear
combination of $e_1, \cdots, e_{n-1}$.  Hence $C$ fixes $v'$ and takes
$-ce_n$  to $-c(e_n+v)=-ce_n-cv$.  Therefore,  $CA^{-k}v= A^{-k}v-cv$,
i.e., $Cw_2=w_2-cw_1$.  Moreover, $Cv=v$,  i.e.,  $Cw_1=w_1$.  \\

The matrix form of $A^{-k}CA^k$ is similarly determined. \\ 

If $\mid c \mid \leq 2$, it  is well known (for a reference see 
Theorem (3.1) of \cite{Kcon}) that the two matrices  $M_1$ and $M_2$ 
generate a  subgroup of finite index in $\SL_2(\Z)$.
\end{proof}

\subsection{Proof   of   Theorem   \ref{mainth}}  In   the   preceding
subsection, we  have already  verified that the  vectors $w_1,w_2,w_3$
are linearly  independent, and that $\Omega (w_1,w_2)\neq  0$. We will
now verify  that the group  generated by the transvections  $C_1, C_2,
C_3$ satisfies the hypotheses of Theorem \ref{criterion}. \\

We  need only  check  that  the group  $\mathrm{H}$  generated by  the
restrictions  $C_i'$  to $W$  of  the  $C_i$,  contains a  non-trivial
element of the unipotent  radical of $\Sp_W$.  Let $W\cap W^{\perp}=\Q
e$ for some $e\in W\setminus  \{0\}$. Denote by $w_1',w_2'$ the images
of  $w_1,w_2$ under  the quotient  map $W\rightarrow  W/\Q  e$.  Since
$\Omega  (w_1,w_2)\neq  0$, it  follows  that  the  span of  $w_1,w_2$
intersects trivially  with $\Q e$.  Let $C_1''$, $C_2''$,  and $C_3''$
denote the maps induced on the quotient $W/\Q e$, respectively, by the
transformations $C_1'$, $C_2'$ and $C_3'$. \\

By  Lemma \ref{2by2matrix},  the group  generated by  $C_1'',C_2''$ in
$\GL(W/\Q e)$ is an  arithmetic group $\mathrm{D}$.  In particular, if
$u\in \GL(W/\Q e)$ is a unipotent element, then some power of $u$ lies
in   the  arithmetic  group   $\mathrm{D}$.   Therefore,   some  power
$(C_3'')^m$ of the unipotent element $C_3''$ lies in $\mathrm{D}$. \\

With  respect to  the basis  $e, w_1,w_2$,  the restriction  $C_3'$ of
$C_3$ to $W$ has the matrix form
\[C_3'=
\begin{pmatrix}1  & w \\ 0_2 &  C_3''\end{pmatrix}
= \begin{pmatrix}1  & w \\ 0_2 &  g \end{pmatrix},\] where $w
\in \Q^2$  may be  viewed as  an element of  the unipotent  radical of
$\Sp_W$  and  $0_2$  is  the  zero  column  vector  of  length  $2$.  A
computation shows that
\[(C_3')    ^m=    \begin{pmatrix}    1     &    w_m    \\    0_2    &
g^m\end{pmatrix},\]  where  $w_m=  w(1+g+\cdots +g^{m-1})\in  \Q ^2$
with $g=C_3''$ ($w$ is a $1\times  2$ matrix and we multiply it on the
right by the $2\times 2$-matrix $1+g+\cdots +g^{m-1}$ to get $w_m$). \\

Since $\Omega (w_1, w_3)=\Omega  (v, A^kv)= \Omega (A^{-k}v,v)\neq 0$,
it  follows that the  image $(C_3-1)  w_1$ is  a non-zero  multiple of
$w_3$; since  $w_3$ is not a  linear combination of  $w_1,w_2 $ (Lemma
\ref{independence}),  it follows  that $C_3'(w_1)=x_0e+x_1w_1+x_2w_2$,
with $x_0\neq 0$. Therefore, the vector $w\in \Q ^2$ is non-zero. \\

Since $C_3'$ is unipotent, so is $g=C_3''$; hence $1+g+\cdots +g^{m-1}=
\prod (g-  \omega )$ (where  the product is over  all $\omega $  which are
non-trivial  $m$-th roots  of  unity), is  non-singular  and hence  $w
(1+g+\cdots +g^{m-1})\neq  0$. \\ 

Recall that $m$ was chosen so  that the group generated by $C_1''$ and
$C_2''$ contains $g^m$. Hence the group generated by $C_1'$ and $C_2'$
contains an element of the form
\[h= \begin{pmatrix} 1  & 0 \\ 0_2 &  g^m\end{pmatrix},\] where $0$ in
the first row is the zero row vector in $\Q^2$.

Therefore, multiplying $(C_3')^m$ on the right by the element $h^{-1}$
we get that $(C_3')^mh^{-1}$ has the matrix form
\[(C_3')^mh^{-1}=   \begin{pmatrix}   1   &    w_m   g^{-m}   \\   0_2   &
1  \end{pmatrix}.\] This  clearly  lies in  the  unipotent radical  of
$\Sp_W$. It  is non-trivial:  since $w_m$ is  non-zero and  $g^{-m}$ is
non-singular, the row vector $w_mg^{-m}$  is non-zero. \\ 

Finally,  we  must   check  that  the  group  $\Gamma   $  of  Theorem
\ref{mainth} is Zariski dense; however, this is known by \cite{BH}. \\

Therefore all the conditions of Theorem \ref{criterion} are satisfied,
and  hence,  by  Theorem  \ref{criterion},  the  hypergeometric  group
$\Gamma $ is arithmetic.  This proves Theorem \ref{mainth}.

\section{Remarks} \label{remarks}

\begin{remark}  \label{arithmetictable}  Computations  show that  when
$n=4$, in a large number  of cases, the polynomial $h=f-g$ has leading
coefficient $\leq  2$ (assuming $A,B$ are  quasi-unipotent).  In these
cases,  by Theorem  \ref{mainth},  the monodromy  group  $\Gamma $  is
arithmetic. \\

In several  other cases, $\frac{f-g}{X}=3X^2-X+3$.   Although, in this
case the leading coefficient is  $3$, the monodromy group $\Gamma $ is
still   arithmetic   (see   Example  \ref{bigcoefficient}   of   Table
\ref{table:arithmeticity  not known});  it is  shown by  replacing the
condition of  Theorem \ref{mainth}, with  the condition that  for some
$g\in \Gamma  $, the coefficient of $e_n$  in $gv$ is $\pm  2$ or $\pm
1$.  The proof proceeds in exactly the same way, by replacing $v,A^kv,
A^{-k}v$ by the vectors $v,gv, g^{-1}v$. \\

By  similar  arguments, the  groups  $\Gamma $  (as  we  have seen  in
\ref{example2}, \ref{example12}) in Example \ref{arithmeticYYCE-2} and
Example  \ref{arithmeticYYCE-3}  of  Table  \ref{table:calabiyau}  are
arithmetic groups, though they do not satisfy the condition of Theorem
\ref{mainth}.
\end{remark}

In the following  subsection we give tables for  the polynomials $f,g$
when $n=4$, $(f,g)$ form a primitive pair, $f(0)=g(0)=1$ and $A,B$ are
quasi-unipotent    matrices     with    integral    entries.     Table
\ref{table:arithmeticity  known}  is   the  list  of  examples  where
arithmeticity    follows    from    Theorem    \ref{mainth}.     Table
\ref{table:arithmeticity not known} is  the list of examples where the
hypothesis  $\mid c  \mid \leq  2$  of Theorem  \ref{mainth} does  not
hold. \\

In  table  \ref{table:calabiyau},  we  list  the  pairs  $(f,g)$  with
$f=(X-1)^4$. This corresponds to the list of 14 families of Calabi Yau
threefolds fibreing over $S={\mathbb P}^1\setminus\{0,1,\infty\}$.  we
also specify  (displayed in bold-face  in Table \ref{table:calabiyau})
the  cases when we  are able  to prove  that the  associated monodromy
group  is arithmetic,  using Theorem  \ref{criterion}.  Those  with an
asterisk  {\it and}  bold face  are  the examples  of \cite{BT}  where
$\Gamma$ is proved  to be {\it thin} (and  necessarily, the hypothesis
of Theorem \ref{mainth}  does not hold). In the  remaining cases it is
not known whether the monodromy  is arithmetic or not and we represent
these with a question mark in the last column.

\pagebreak

\renewcommand{\arraystretch}{1.5}   

{\tiny\begin{table}[h]
\addtolength{\tabcolsep}{-2pt}
\caption{List of primitive {\it{Symplectic}} pairs of polynomials of degree 4 (which are  products of cyclotomic polynomials), for which arithmeticity follows from Theorem \ref{mainth}}
\newcounter{rownum}
\setcounter{rownum}{0}
\centering
\begin{tabular}{ |c|  c|   c| c| c| c|}
\hline

    No. & $f(X)$ & $g(X)$ & $\alpha$ & $\beta$ & $f(X)-g(X)$ \\ \hline
   
 \refstepcounter{rownum}\arabic{rownum}& $  X^4-4X^3+6X^2-4X+1$ & $ X^4-2X^3+3X^2-2X+1 $ &   0,0,0,0 & $  
\frac{1}{6}$,$ \frac{1}{6}$,$ \frac{5}{6}$,$ \frac{5}{6} $ & $ -2X^3+3X^2-2X$ \\ \hline
     
\addtocounter{rownum}{1}\arabic{rownum} & $X^4-2X^2+1$& $X^4+2X^3+3X^2+2X+1$ &0,0,$\frac{1}{2}$,$\frac{1}{2}$ &$\frac{1}{3}$,$\frac{1}{3}$,$\frac{2}{3}$,$\frac{2}{3}$ &$-2X^3-5X^2-2X$  \\ \hline 
     
\addtocounter{rownum}{1}\arabic{rownum} & $X^4-2X^2+1$ &$X^4+X^3+2X^2+X+1$  &0,0,$\frac{1}{2}$,$\frac{1}{2}$ &$\frac{1}{4}$,$\frac{3}{4}$,$\frac{1}{3}$,$\frac{2}{3}$ & $-X^3-4X^2-X$  \\ \hline
     
\addtocounter{rownum}{1}\arabic{rownum} & $X^4-2X^2+1$ &$X^4+X^3+X^2+X+1$ &0,0,$\frac{1}{2}$,$\frac{1}{2}$  & $\frac{1}{5}$,$\frac{2}{5}$,$\frac{3}{5}$,$\frac{4}{5}$& $-X^3-3X^2-X$ \\ \hline
    
\addtocounter{rownum}{1}\arabic{rownum} & $X^4-2X^2+1$ &$X^4-2X^3+3X^2-2X+1$ &0,0,$\frac{1}{2}$,$\frac{1}{2}$  &$\frac{1}{6}$,$\frac{1}{6}$,$\frac{5}{6}$,$\frac{5}{6}$ &$2X^3-5X^2+2X$ \\ \hline
    
\addtocounter{rownum}{1}\arabic{rownum} & $X^4-2X^2+1$ &$X^4-X^3+2X^2-X+1$  &0,0,$\frac{1}{2}$,$\frac{1}{2}$ &$\frac{1}{4}$,$\frac{3}{4}$,$\frac{1}{6}$,$\frac{5}{6}$ &$X^3-4X^2+X$ \\ \hline
    
\addtocounter{rownum}{1}\arabic{rownum} & $X^4-2X^2+1$ &$X^4-X^3+X^2-X+1$  &0,0,$\frac{1}{2}$,$\frac{1}{2}$ &$\frac{1}{10}$,$\frac{3}{10}$,$\frac{7}{10}$,$\frac{9}{10}$ &$X^3-3X^2+X$ \\ \hline
    
\addtocounter{rownum}{1}\arabic{rownum} & $X^4+4X^3+6X^2+4X+1$ &$X^4+2X^3+3X^2+2X+1$  &$\frac{1}{2}$,$\frac{1}{2}$,$\frac{1}{2}$,$\frac{1}{2}$ &$\frac{1}{3}$,$\frac{1}{3}$,$\frac{2}{3}$,$\frac{2}{3}$ &$2X^3+3X^2+2X$ \\ \hline
    
\addtocounter{rownum}{1}\arabic{rownum} & $X^4-X^3-X+1$ &$X^4+2X^2+1$  &0,0,$\frac{1}{3}$,$\frac{2}{3}$ &$\frac{1}{4}$,$\frac{1}{4}$,$\frac{3}{4}$,$\frac{3}{4}$ &$-X^3-2X^2-X$ \\ \hline
    
\addtocounter{rownum}{1}\arabic{rownum} & $X^4-X^3-X+1$ &$X^4+X^3+X^2+X+1$  &0,0,$\frac{1}{3}$,$\frac{2}{3}$ &$\frac{1}{5}$,$\frac{2}{5}$,$\frac{3}{5}$,$\frac{4}{5}$ &$-2X^3-X^2-2X$ \\ \hline
    
\addtocounter{rownum}{1}\arabic{rownum} & $X^4-X^3-X+1$ &$X^4-2X^3+3X^2-2X+1$  &0,0,$\frac{1}{3}$,$\frac{2}{3}$ &$\frac{1}{6}$,$\frac{1}{6}$,$\frac{5}{6}$,$\frac{5}{6}$ &$X^3-3X^2+X$ \\ \hline
     
\addtocounter{rownum}{1}\arabic{rownum} & $X^4-X^3-X+1$ &$X^4+X^3+X+1$  &0,0,$\frac{1}{3}$,$\frac{2}{3}$ &$\frac{1}{2}$,$\frac{1}{2}$,$\frac{1}{6}$,$\frac{5}{6}$ &$-2X^3-2X$ \\ \hline
     
\addtocounter{rownum}{1}\arabic{rownum} & $X^4-X^3-X+1$ &$X^4-X^3+2X^2-X+1$  &0,0,$\frac{1}{3}$,$\frac{2}{3}$ &$\frac{1}{4}$,$\frac{3}{4}$,$\frac{1}{6}$,$\frac{5}{6}$ &$-2X^2$ \\ \hline
     
\addtocounter{rownum}{1}\arabic{rownum} & $X^4-X^3-X+1$ &$X^4+1$  &0,0,$\frac{1}{3}$,$\frac{2}{3}$ &$\frac{1}{8}$,$\frac{3}{8}$,$\frac{5}{8}$,$\frac{7}{8}$ &$-X^3-X$ \\ \hline
     
\addtocounter{rownum}{1}\arabic{rownum} & $X^4-X^3-X+1$ &$X^4-X^3+X^2-X+1$  &0,0,$\frac{1}{3}$,$\frac{2}{3}$ &$\frac{1}{10}$,$\frac{3}{10}$ ,$\frac{7}{10}$,$\frac{9}{10}$ &$-X^2$ \\ \hline
    
\addtocounter{rownum}{1}\arabic{rownum} & $X^4-X^3-X+1$ &$X^4-X^2+1$  &0,0,$\frac{1}{3}$,$\frac{2}{3}$ &$\frac{1}{12}$,$\frac{5}{12}$,$\frac{7}{12}$,$\frac{11}{12}$ &$-X^3+X^2-X$ \\ \hline
            
\addtocounter{rownum}{1}\arabic{rownum} & $X^4+2X^3+3X^2+2X+1$ &$X^4+2X^2+1$  &$\frac{1}{3}$,$\frac{1}{3}$,$\frac{2}{3}$,$\frac{2}{3}$ &$\frac{1}{4}$,$\frac{1}{4}$,$\frac{3}{4}$,$\frac{3}{4}$ &$2X^3+X^2+2X$ \\ \hline
    
\addtocounter{rownum}{1}\arabic{rownum} & $X^4+2X^3+3X^2+2X+1$ &$X^4+2X^3+2X^2+2X+1$  &$\frac{1}{3}$,$\frac{1}{3}$,$\frac{2}{3}$,$\frac{2}{3}$ &$\frac{1}{2}$,$\frac{1}{2}$,$\frac{1}{4}$,$\frac{3}{4}$ &$X^2$ \\ \hline
    
\addtocounter{rownum}{1}\arabic{rownum} & $X^4+2X^3+3X^2+2X+1$ &$X^4+X^3+X^2+X+1$  &$\frac{1}{3}$,$\frac{1}{3}$,$\frac{2}{3}$,$\frac{2}{3}$ &$\frac{1}{5}$,$\frac{2}{5}$,$\frac{3}{5}$,$\frac{4}{5}$ &$X^3+2X^2+X$ \\ \hline
    
\addtocounter{rownum}{1}\arabic{rownum} & $X^4+2X^3+3X^2+2X+1$ &$X^4+X^3+X+1$  &$\frac{1}{3}$,$\frac{1}{3}$,$\frac{2}{3}$,$\frac{2}{3}$ &$\frac{1}{2}$,$\frac{1}{2}$,$\frac{1}{6}$,$\frac{5}{6}$ &$X^3+3X^2+X$ \\ \hline
    
\addtocounter{rownum}{1}\arabic{rownum} & $X^4+2X^3+3X^2+2X+1$ &$X^4+1$  &$\frac{1}{3}$,$\frac{1}{3}$,$\frac{2}{3}$,$\frac{2}{3}$ &$\frac{1}{8}$,$\frac{3}{8}$,$\frac{5}{8}$,$\frac{7}{8}$ &$2X^3+3X^2+2X$ \\ \hline
    
\addtocounter{rownum}{1}\arabic{rownum} & $X^4+2X^3+3X^2+2X+1$ &$X^4-X^2+1$  &$\frac{1}{3}$,$\frac{1}{3}$,$\frac{2}{3}$,$\frac{2}{3}$ &$\frac{1}{12}$,$\frac{5}{12}$,$\frac{7}{12}$,$\frac{11}{12}$ &$2X^3+4X^2+2X$ \\ \hline
    
\addtocounter{rownum}{1}\arabic{rownum} & $X^4+3X^3+4X^2+3X+1$ &$X^4+X^3+X^2+X+1$  &$\frac{1}{2}$,$\frac{1}{2}$,$\frac{1}{3}$,$\frac{2}{3}$ &$\frac{1}{5}$,$\frac{2}{5}$,$\frac{3}{5}$,$\frac{4}{5}$ &$2X^3+3X^2+2X$ \\ \hline
    
\addtocounter{rownum}{1}\arabic{rownum} & $X^4-2X^3+2X^2-2X+1$ &$X^4-2X^3+3X^2-2X+1$  &0,0,$\frac{1}{4}$,$\frac{3}{4}$ &$\frac{1}{6}$,$\frac{1}{6}$,$\frac{5}{6}$,$\frac{5}{6}$ &$-X^2$ \\ \hline
    
\addtocounter{rownum}{1}\arabic{rownum} & $X^4-2X^3+2X^2-2X+1$ &$X^4+X^2+1$  &0,0,$\frac{1}{4}$,$\frac{3}{4}$ &$\frac{1}{3}$,$\frac{2}{3}$,$\frac{1}{6}$,$\frac{5}{6}$ &$-2X^3+X^2-2X$ \\ \hline
    
\addtocounter{rownum}{1}\arabic{rownum} & $X^4-2X^3+2X^2-2X+1$ &$X^4+1$  &0,0,$\frac{1}{4}$,$\frac{3}{4}$ &$\frac{1}{8}$,$\frac{3}{8}$,$\frac{5}{8}$,$\frac{7}{8}$ &$-2X^3+2X^2-2X$ \\ \hline

\addtocounter{rownum}{1}\arabic{rownum} & $X^4-2X^3+2X^2-2X+1$ &$X^4-X^3+X^2-X+1$  &0,0,$\frac{1}{4}$,$\frac{3}{4}$ &$\frac{1}{10}$,$\frac{3}{10}$,$\frac{7}{10}$,$\frac{9}{10}$ &$-X^3+X^2-X$ \\ \hline
    
\addtocounter{rownum}{1}\arabic{rownum} & $X^4-2X^3+2X^2-2X+1$ &$X^4-X^2+1$  &0,0,$\frac{1}{4}$,$\frac{3}{4}$ &$\frac{1}{12}$,$\frac{5}{12}$,$\frac{7}{12}$,$\frac{11}{12}$ &$-2X^3+3X^2-2X$ \\ \hline
    
\addtocounter{rownum}{1}\arabic{rownum} & $X^4+2X^2+1$ &$X^4+X^3+X^2+X+1$  &$\frac{1}{4}$,$\frac{1}{4}$,$\frac{3}{4}$,$\frac{3}{4}$ &$\frac{1}{5}$,$\frac{2}{5}$,$\frac{3}{5}$,$\frac{4}{5}$ &$-X^3+X^2-X$ \\ \hline   
    
\addtocounter{rownum}{1}\arabic{rownum} & $X^4+2X^2+1$ &$X^4-2X^3+3X^2-2X+1$  &$\frac{1}{4}$,$\frac{1}{4}$,$\frac{3}{4}$,$\frac{3}{4}$ &$\frac{1}{6}$,$\frac{1}{6}$,$\frac{5}{6}$,$\frac{5}{6}$ &$2X^3-X^2+2X$ \\ \hline
         \end{tabular}
\label{table:arithmeticity known}

Table \ref{table:arithmeticity known} continued...
\end{table}}
 
 \pagebreak

{\tiny\begin{table}[h]
Table \ref{table:arithmeticity known} continued...
\addtolength{\tabcolsep}{-2pt}

  \centering
\begin{tabular}{ |c|  c|   c| c| c| c|}
    \hline
    No. & $f(X)$ & $g(X)$ & $\alpha$ & $\beta$ & $f(X)-g(X)$  \\ \hline

   \addtocounter{rownum}{1}\arabic{rownum} & $X^4+2X^2+1$ &$X^4+X^3+X+1$  &$\frac{1}{4}$,$\frac{1}{4}$,$\frac{3}{4}$,$\frac{3}{4}$ &$\frac{1}{2}$,$\frac{1}{2}$,$\frac{1}{6}$,$\frac{5}{6}$ &$-X^3+2X^2-X$ \\ \hline
    
    \addtocounter{rownum}{1}\arabic{rownum} & $X^4+2X^2+1$ &$X^4-X^3+X^2-X+1$  &$\frac{1}{4}$,$\frac{1}{4}$,$\frac{3}{4}$,$\frac{3}{4}$ &$\frac{1}{10}$,$\frac{3}{10}$,$\frac{7}{10}$,$\frac{9}{10}$ &$X^3+X^2+X$ \\ \hline
     
    \addtocounter{rownum}{1}\arabic{rownum} & $X^4+2X^3+2X^2+2X+1$ &$X^4+X^3+X^2+X+1$  &$\frac{1}{2}$,$\frac{1}{2}$,$\frac{1}{4}$,$\frac{3}{4}$ &$\frac{1}{5}$,$\frac{2}{5}$,$\frac{3}{5}$,$\frac{4}{5}$ &$X^3+X^2+X$ \\ \hline
     
    \addtocounter{rownum}{1}\arabic{rownum} & $X^4+2X^3+2X^2+2X+1$ &$X^4+X^2+1$  &$\frac{1}{2}$,$\frac{1}{2}$,$\frac{1}{4}$,$\frac{3}{4}$ &$\frac{1}{3}$,$\frac{2}{3}$,$\frac{1}{6}$,$\frac{5}{6}$ &$2X^3+X^2+2X$ \\ \hline
     
    \addtocounter{rownum}{1}\arabic{rownum} & $X^4+2X^3+2X^2+2X+1$ &$X^4+1$  &$\frac{1}{2}$,$\frac{1}{2}$,$\frac{1}{4}$,$\frac{3}{4}$ &$\frac{1}{8}$,$\frac{3}{8}$,$\frac{5}{8}$,$\frac{7}{8}$ &$2X^3+2X^2+2X$ \\ \hline
     
    \addtocounter{rownum}{1}\arabic{rownum} & $X^4+2X^3+2X^2+2X+1$ &$X^4-X^2+1$  &$\frac{1}{2}$,$\frac{1}{2}$,$\frac{1}{4}$,$\frac{3}{4}$ &$\frac{1}{12}$,$\frac{5}{12}$,$\frac{7}{12}$,$\frac{11}{12}$ &$2X^3+3X^2+2X$ \\ \hline
     
    \addtocounter{rownum}{1}\arabic{rownum} & $X^4+X^3+2X^2+X+1$ &$X^4+X^3+X^2+X+1$  &$\frac{1}{3}$,$\frac{2}{3}$,$\frac{1}{4}$,$\frac{3}{4}$ &$\frac{1}{5}$,$\frac{2}{5}$,$\frac{3}{5}$,$\frac{4}{5}$ &$X^2$ \\ \hline
        
    \addtocounter{rownum}{1}\arabic{rownum} & $X^4+X^3+2X^2+X+1$ &$X^4+X^3+X+1$  &$\frac{1}{3}$,$\frac{2}{3}$,$\frac{1}{4}$,$\frac{3}{4}$ &$\frac{1}{2}$,$\frac{1}{2}$,$\frac{1}{6}$,$\frac{5}{6}$ &$2X^2$ \\ \hline
     
    \addtocounter{rownum}{1}\arabic{rownum} & $X^4+X^3+2X^2+X+1$ &$X^4+1$  &$\frac{1}{3}$,$\frac{2}{3}$,$\frac{1}{4}$,$\frac{3}{4}$ &$\frac{1}{8}$,$\frac{3}{8}$,$\frac{5}{8}$,$\frac{7}{8}$ &$X^3+2X^2+X$ \\ \hline
     
    \addtocounter{rownum}{1}\arabic{rownum} & $X^4+X^3+2X^2+X+1$ &$X^4-X^3+X^2-X+1$  &$\frac{1}{3}$,$\frac{2}{3}$,$\frac{1}{4}$,$\frac{3}{4}$ &$\frac{1}{10}$,$\frac{3}{10}$,$\frac{7}{10}$,$\frac{9}{10}$ &$2X^3+X^2+2X$ \\ \hline
    
    \addtocounter{rownum}{1}\arabic{rownum} & $X^4+X^3+2X^2+X+1$ &$X^4-X^2+1$  &$\frac{1}{3}$,$\frac{2}{3}$,$\frac{1}{4}$,$\frac{3}{4}$ &$\frac{1}{12}$,$\frac{5}{12}$,$\frac{7}{12}$,$\frac{11}{12}$ &$X^3+3X^2+X$ \\ \hline
    
    \addtocounter{rownum}{1}\arabic{rownum} & $X^4+X^3+X^2+X+1$ &$X^4+X^3+X+1$  &$\frac{1}{5}$,$\frac{2}{5}$,$\frac{3}{5}$,$\frac{4}{5}$ &$\frac{1}{2}$,$\frac{1}{2}$,$\frac{1}{6}$,$\frac{5}{6}$ &$X^2$ \\ \hline
    
    \addtocounter{rownum}{1}\arabic{rownum} & $X^4+X^3+X^2+X+1$ &$X^4+X^2+1$  &$\frac{1}{5}$,$\frac{2}{5}$,$\frac{3}{5}$,$\frac{4}{5}$ &$\frac{1}{3}$,$\frac{2}{3}$,$\frac{1}{6}$,$\frac{5}{6}$ &$X^3+X$ \\ \hline
    
    \addtocounter{rownum}{1}\arabic{rownum} & $X^4+X^3+X^2+X+1$ &$X^4-X^3+2X^2-X+1$  &$\frac{1}{5}$,$\frac{2}{5}$,$\frac{3}{5}$,$\frac{4}{5}$ &$\frac{1}{4}$,$\frac{3}{4}$,$\frac{1}{6}$,$\frac{5}{6}$ &$2X^3-X^2+2X$ \\ \hline
    
    \addtocounter{rownum}{1}\arabic{rownum} & $X^4+X^3+X^2+X+1$ &$X^4+1$  &$\frac{1}{5}$,$\frac{2}{5}$,$\frac{3}{5}$,$\frac{4}{5}$ &$\frac{1}{8}$,$\frac{3}{8}$,$\frac{5}{8}$,$\frac{7}{8}$ &$X^3+X^2+X$ \\ \hline
    
    \addtocounter{rownum}{1}\arabic{rownum} & $X^4+X^3+X^2+X+1$ &$X^4-X^3+X^2-X+1$  &$\frac{1}{5}$,$\frac{2}{5}$,$\frac{3}{5}$,$\frac{4}{5}$ &$\frac{1}{10}$,$\frac{3}{10}$,$\frac{7}{10}$,$\frac{9}{10}$ &$2X^3+2X$ \\ \hline
    
    \addtocounter{rownum}{1}\arabic{rownum} & $X^4+X^3+X^2+X+1$ &$X^4-X^2+1$  &$\frac{1}{5}$,$\frac{2}{5}$,$\frac{3}{5}$,$\frac{4}{5}$ &$\frac{1}{12}$,$\frac{5}{12}$,$\frac{7}{12}$,$\frac{11}{12}$ &$X^3+2X^2+X$ \\ \hline
    
    \addtocounter{rownum}{1}\arabic{rownum} & $X^4-3X^3+4X^2-3X+1$ &$X^4-X^3+X^2-X+1$  &0,0,$\frac{1}{6}$,$\frac{5}{6}$ &$\frac{1}{10}$,$\frac{3}{10}$,$\frac{7}{10}$,$\frac{9}{10}$ &$-2X^3+3X^2-2X$ \\ \hline
    
    \addtocounter{rownum}{1}\arabic{rownum} & $X^4-2X^3+3X^2-2X+1$ &$X^4+1$  &$\frac{1}{6}$,$\frac{1}{6}$,$\frac{5}{6}$,$\frac{5}{6}$ &$\frac{1}{8}$,$\frac{3}{8}$,$\frac{5}{8}$,$\frac{7}{8}$ &$-2X^3+3X^2-2X$ \\ \hline
    
    \addtocounter{rownum}{1}\arabic{rownum} & $X^4-2X^3+3X^2-2X+1$ &$X^4-X^3+X^2-X+1$  &$\frac{1}{6}$,$\frac{1}{6}$,$\frac{5}{6}$,$\frac{5}{6}$ &$\frac{1}{10}$,$\frac{3}{10}$,$\frac{7}{10}$,$\frac{9}{10}$ &$-X^3+2X^2-X$ \\ \hline
    
    \addtocounter{rownum}{1}\arabic{rownum} & $X^4-2X^3+3X^2-2X+1$ &$X^4-X^2+1$  &$\frac{1}{6}$,$\frac{1}{6}$,$\frac{5}{6}$,$\frac{5}{6}$ &$\frac{1}{12}$,$\frac{5}{12}$,$\frac{7}{12}$,$\frac{11}{12}$ &$-2X^3+4X^2-2X$ \\ \hline
    
    \addtocounter{rownum}{1}\arabic{rownum} & $X^4+X^3+X+1$ &$X^4+1$  &$\frac{1}{2}$,$\frac{1}{2}$,$\frac{1}{6}$,$\frac{5}{6}$ &$\frac{1}{8}$,$\frac{3}{8}$,$\frac{5}{8}$,$\frac{7}{8}$ &$X^3+X$ \\ \hline
     
    \addtocounter{rownum}{1}\arabic{rownum} & $X^4+X^3+X+1$ &$X^4-X^3+X^2-X+1$  &$\frac{1}{2}$,$\frac{1}{2}$,$\frac{1}{6}$,$\frac{5}{6}$ &$\frac{1}{10}$,$\frac{3}{10}$,$\frac{7}{10}$,$\frac{9}{10}$ &$2X^3-X^2+2X$ \\ \hline
     
    \addtocounter{rownum}{1}\arabic{rownum} & $X^4+X^3+X+1$ &$X^4-X^2+1$  &$\frac{1}{2}$,$\frac{1}{2}$,$\frac{1}{6}$,$\frac{5}{6}$ &$\frac{1}{12}$,$\frac{5}{12}$,$\frac{7}{12}$,$\frac{11}{12}$ &$X^3+X^2+X$ \\ \hline
     
    \addtocounter{rownum}{1}\arabic{rownum} & $X^4+X^2+1$ &$X^4-X^3+X^2-X+1$  &$\frac{1}{3}$,$\frac{2}{3}$,$\frac{1}{6}$,$\frac{5}{6}$ &$\frac{1}{10}$,$\frac{3}{10}$,$\frac{7}{10}$,$\frac{9}{10}$ &$X^3+X$ \\ \hline
     
    \addtocounter{rownum}{1}\arabic{rownum} & $X^4-X^3+2X^2-X+1$ &$X^4+1$  &$\frac{1}{4}$,$\frac{3}{4}$,$\frac{1}{6}$,$\frac{5}{6}$ &$\frac{1}{8}$,$\frac{3}{8}$,$\frac{5}{8}$,$\frac{7}{8}$ &$-X^3+2X^2-X$ \\ \hline
    
    \addtocounter{rownum}{1}\arabic{rownum} & $X^4-X^3+2X^2-X+1$ &$X^4-X^3+X^2-X+1$  &$\frac{1}{4}$,$\frac{3}{4}$,$\frac{1}{6}$,$\frac{5}{6}$ &$\frac{1}{10}$,$\frac{3}{10}$,$\frac{7}{10}$,$\frac{9}{10}$ &$X^2$ \\ \hline
    
    \addtocounter{rownum}{1}\arabic{rownum} & $X^4-X^3+2X^2-X+1$ &$X^4-X^2+1$  &$\frac{1}{4}$,$\frac{3}{4}$,$\frac{1}{6}$,$\frac{5}{6}$ &$\frac{1}{12}$,$\frac{5}{12}$,$\frac{7}{12}$,$\frac{11}{12}$ &$-X^3+3X^2-X$ \\ \hline
    
    \addtocounter{rownum}{1}\arabic{rownum} & $X^4+1$ &$X^4-X^3+X^2-X+1$  &$\frac{1}{8}$,$\frac{3}{8}$,$\frac{5}{8}$,$\frac{7}{8}$ &$\frac{1}{10}$,$\frac{3}{10}$,$\frac{7}{10}$,$\frac{9}{10}$ &$X^3-X^2+X$ \\ \hline
    
    \addtocounter{rownum}{1}\arabic{rownum} & $X^4-X^3+X^2-X+1$ &$X^4-X^2+1$  &$\frac{1}{10}$,$\frac{3}{10}$,$\frac{7}{10}$,$\frac{9}{10}$ &$\frac{1}{12}$,$\frac{5}{12}$,$\frac{7}{12}$,$\frac{11}{12}$ &$-X^3+2X^2-X$ \\ \hline
    
 \end{tabular}
 \end{table}}

\pagebreak

{\tiny\begin{table}[h]
\addtolength{\tabcolsep}{-2pt}
\caption{List of primitive {\it{Symplectic}} pairs of polynomials of degree 4 (which are  products of cyclotomic polynomials), 
to which Theorem \ref{mainth} does not apply}

\newcounter{rownum-2}
\setcounter{rownum-2}{0}
\centering
\begin{tabular}{ |c|  c|   c| c| c| c|}
\hline

  No. & $f(X)$ & $g(X)$ & $\alpha$ & $\beta$ & $f(X)-g(X)$ \\ \hline
  \refstepcounter{rownum-2}\arabic{rownum-2} & $ X^4-4X^3+6X^2-4X+1$ &$ X^4+4X^3+6X^2+4X+1$  &0,0,0,0 &$ \frac{1}{2}$,$\frac{1}{2}$,$\frac{1}{2}$,$\frac{1}{2}$ &$-8X^3-8X$ \\ \hline

 \addtocounter{rownum-2}{1}\arabic{rownum-2} & $X^4-4X^3+6X^2-4X+1$ &$X^4+2X^3+3X^2+2X+1$  &0,0,0,0 &$\frac{1}{3}$,$\frac{1}{3}$,$\frac{2}{3}$,$\frac{2}{3}$ &$-6X^3+3X^2-6X$ \\ \hline
  
\refstepcounter{rownum-2}\arabic{rownum-2} &  $ X^4-4X^3+6X^2-4X+1$ &$ X^4+3X^3+4X^2+3X+1$  & 0,0,0,0 &$\frac{1}{2}$,$\frac{1}{2}$,$\frac{1}{3}$,$\frac{2}{3}$ &$-7X^3+2X^2-7X$ \\ \hline
  
  \addtocounter{rownum-2}{1}\arabic{rownum-2} & $X^4-4X^3+6X^2-4X+1$ &$X^4+2X^2+1$  &0,0,0,0 &$\frac{1}{4}$,$\frac{1}{4}$,$\frac{3}{4}$,$\frac{3}{4}$ &$-4X^3+4X^2-4X$ \\ \hline
  
\refstepcounter{rownum-2}\arabic{rownum-2} & $X^4-4X^3+6X^2-4X+1$ &$X^4+2X^3+2X^2+2X+1$  &0,0,0,0 &$\frac{1}{2}$,$\frac{1}{2}$,$\frac{1}{4}$,$\frac{3}{4}$ &$-6X^3+4X^2-6X$ \\ \hline
  
\addtocounter{rownum-2}{1}\arabic{rownum-2} & $X^4-4X^3+6X^2-4X+1$ &$X^4+X^3+2X^2+X+1$  &0,0,0,0 &$\frac{1}{3}$,$\frac{2}{3}$,$\frac{1}{4}$,$\frac{3}{4}$ &$-5X^3+4X^2-5X$ \\ \hline
  
  \refstepcounter{rownum-2}\arabic{rownum-2} & $X^4-4X^3+6X^2-4X+1$ & $X^4+X^3+X^2+X+1$  &0,0,0,0 &$
\frac{1}{5}$,$\frac{2}{5}$,$\frac{3}{5}$,$\frac{4}{5}$ &$-5X^3+5X^2-5X$ \\ \hline
  
 \refstepcounter{rownum-2}\arabic{rownum-2} & $X^4-4X^3+6X^2-4X+1$ &$X^4+X^3+X+1$  &0,0,0,0 &$\frac{1}{2}$,$\frac{1}{2}$,$\frac{1}{6}$,$\frac{5}{6}$ &$-5X^3+6X^2-5X$ \\ \hline
  
  \addtocounter{rownum-2}{1}\arabic{rownum-2} & $X^4-4X^3+6X^2-4X+1$ &$X^4+X^2+1$  &0,0,0,0 &$\frac{1}{3}$,$\frac{2}{3}$,$\frac{1}{6}$,$\frac{5}{6}$ &$-4X^3+5X^2-4X$ \\ \hline
  
  \refstepcounter{rownum-2}\arabic{rownum-2} & $X^4-4X^3+6X^2-4X+1$ &$X^4-X^3+2X^2-X+1$  &0,0,0,0 &$\frac{1}{4}$,$\frac{3}{4}$,$\frac{1}{6}$,$\frac{5}{6}$ &$-3X^3+4X^2-3X$ \\ \hline
  
  \refstepcounter{rownum-2}\arabic{rownum-2} & $ X^4-4X^3+6X^2-4X+1$ &$ X^4+1$  & 0,0,0,0 &$ \frac{1}{8}$,$ \frac{3}{8}$,$ 
\frac{5}{8}$,$ \frac{7}{8}$ &$ -4X^3+6X^2-4X$ \\ \hline
  
  \refstepcounter{rownum-2}\arabic{rownum-2} & $ X^4-4X^3+6X^2-4X+1$ &$ X^4-X^3+X^2-X+1$  & 0,0,0,0 &$ \frac{1}{10}$,$ \frac{3}{10}$,$ \frac{7}{10}$,$\frac{9}{10}$ &$ -3X^3+5X^2-3X$ \\ \hline
  
   \refstepcounter{rownum-2}\arabic{rownum-2} & $ X^4-4X^3+6X^2-4X+1$ &$ X^4-X^2+1$  &0,0,0,0 &$ \frac{1}{12}$,$ \frac{5}{12}$,$ \frac{7}{12}$,$ \frac{11}{12}$ &$ -4X^3+7X^2-4X$ \\ \hline
  
  \addtocounter{rownum-2}{1}\arabic{rownum-2} & $X^4+4X^3+6X^2+4X+1$ &$X^4-X^3-X+1$  &$\frac{1}{2}$,$\frac{1}{2}$,$\frac{1}{2}$,$\frac{1}{2}$ &0,0,$\frac{1}{3}$,$\frac{2}{3}$ &$5X^3+6X^2+5X$ \\ \hline
  
  \addtocounter{rownum-2}{1}\arabic{rownum-2} & $X^4+4X^3+6X^2+4X+1$ &$X^4-2X^3+2X^2-2X+1$  &$\frac{1}{2}$,$\frac{1}{2}$,$\frac{1}{2}$,$\frac{1}{2}$ &$0,0,\frac{1}{4}$,$\frac{3}{4}$ &$6X^3+4X^2+6X$ \\ \hline
  
  \addtocounter{rownum-2}{1}\arabic{rownum-2} & $X^4+4X^3+6X^2+4X+1$ &$X^4+2X^2+1$  &$\frac{1}{2}$,$\frac{1}{2}$,$\frac{1}{2}$,$\frac{1}{2}$ &$\frac{1}{4}$,$\frac{1}{4}$,$\frac{3}{4}$,$\frac{3}{4}$ &$4X^3+4X^2+4X$ \\ \hline
  
  \addtocounter{rownum-2}{1}\arabic{rownum-2}& $X^4+4X^3+6X^2+4X+1$ &$X^4+X^3+2X^2+X+1$  &$\frac{1}{2}$,$\frac{1}{2}$,$\frac{1}{2}$,$\frac{1}{2}$ &$\frac{1}{4}$,$\frac{3}{4}$,$\frac{1}{3}$,$\frac{2}{3}$ &$3X^3+4X^2+3X$ \\ \hline
  
  \addtocounter{rownum-2}{1}\arabic{rownum-2} & $X^4+4X^3+6X^2+4X+1$ &$X^4+X^3+X^2+X+1$  &$\frac{1}{2}$,$\frac{1}{2}$,$\frac{1}{2}$,$\frac{1}{2}$ &$\frac{1}{5}$,$\frac{2}{5}$,$\frac{3}{5}$,$\frac{4}{5}$ &$3X^3+5X^2+3X$ \\ \hline
  
  \addtocounter{rownum-2}{1}\arabic{rownum-2} & $X^4+4X^3+6X^2+4X+1$ &$X^4-3X^3+4X^2-3X+1$  &$\frac{1}{2}$,$\frac{1}{2}$,$\frac{1}{2}$,$\frac{1}{2}$ &0,0,$\frac{1}{6}$,$\frac{5}{6}$ &$7X^3+2X^2+7X$ \\ \hline
  
  \addtocounter{rownum-2}{1}\arabic{rownum-2} & $X^4+4X^3+6X^2+4X+1$ &$X^4-2X^3+3X^2-2X+1$  &$\frac{1}{2}$,$\frac{1}{2}$,$\frac{1}{2}$,$\frac{1}{2}$ &$\frac{1}{6}$,$\frac{1}{6}$,$\frac{5}{6}$,$\frac{5}{6}$ &$6X^3+3X^2+6X$ \\ \hline
  
  \addtocounter{rownum-2}{1}\arabic{rownum-2} & $X^4+4X^3+6X^2+4X+1$ &$X^4+X^2+1$  &$\frac{1}{2}$,$\frac{1}{2}$,$\frac{1}{2}$,$\frac{1}{2}$ &$\frac{1}{3}$,$\frac{2}{3}$,$\frac{1}{6}$,$\frac{5}{6}$ &$4X^3+5X^2+4X$ \\ \hline
  
  \addtocounter{rownum-2}{1}\arabic{rownum-2} & $X^4+4X^3+6X^2+4X+1$ &$X^4-X^3+2X^2-X+1$  &$\frac{1}{2}$,$\frac{1}{2}$,$\frac{1}{2}$,$\frac{1}{2}$ &$\frac{1}{4}$,$\frac{3}{4}$,$\frac{1}{6}$,$\frac{5}{6}$ &$5X^3+4X^2+5X$ \\ \hline
  
  \addtocounter{rownum-2}{1}\arabic{rownum-2} & $X^4+4X^3+6X^2+4X+1$ &$X^4+1$  &$\frac{1}{2}$,$\frac{1}{2}$,$\frac{1}{2}$,$\frac{1}{2}$ &$\frac{1}{8}$,$\frac{3}{8}$,$\frac{5}{8}$,$\frac{7}{8}$ &$4X^3+6X^2+4X$ \\ \hline
  
  \addtocounter{rownum-2}{1}\arabic{rownum-2} & $X^4+4X^3+6X^2+4X+1$ &$X^4-X^3+X^2-X+1$  &$\frac{1}{2}$,$\frac{1}{2}$,$\frac{1}{2}$,$\frac{1}{2}$ &$\frac{1}{10}$,$\frac{3}{10}$,$\frac{7}{10}$,$\frac{9}{10}$ &$5X^3+5X^2+5X$ \\ \hline
  
  \addtocounter{rownum-2}{1}\arabic{rownum-2} & $X^4+4X^3+6X^2+4X+1$ &$X^4-X^2+1$  &$\frac{1}{2}$,$\frac{1}{2}$,$\frac{1}{2}$,$\frac{1}{2}$ &$\frac{1}{12}$,$\frac{5}{12}$,$\frac{7}{12}$,$\frac{11}{12}$ &$4X^3+7X^2+4X$ \\ \hline
  
  \addtocounter{rownum-2}{1}\arabic{rownum-2} & $X^4-X^3-X+1$ &$X^4+2X^3+2X^2+2X+1$  &0,0,$\frac{1}{3}$,$\frac{2}{3}$ &$\frac{1}{2}$,$\frac{1}{2}$,$\frac{1}{4}$,$\frac{3}{4}$ &$-3X^3-2X^2-3X$ \\ \hline
  
  \addtocounter{rownum-2}{1}\arabic{rownum-2} & $X^4+2X^3+3X^2+2X+1$ &$X^4-2X^3+2X^2-2X+1$  &$\frac{1}{3}$,$\frac{1}{3}$,$\frac{2}{3}$,$\frac{2}{3}$ &0,0,$\frac{1}{4}$,$\frac{3}{4}$ &$4X^3+X^2+4X$ \\ \hline
  
  \addtocounter{rownum-2}{1}\arabic{rownum-2} & $X^4+2X^3+3X^2+2X+1$ &$X^4-3X^3+4X^2-3X+1$  &$\frac{1}{3}$,$\frac{1}{3}$,$\frac{2}{3}$,$\frac{2}{3}$ &0,0,$\frac{1}{6}$,$\frac{5}{6}$ &$5X^3-X^2+5X$ \\ \hline
  
  \addtocounter{rownum-2}{1}\arabic{rownum-2} & $X^4+2X^3+3X^2+2X+1$ &$X^4-2X^3+3X^2-2X+1$  &$\frac{1}{3}$,$\frac{1}{3}$,$\frac{2}{3}$,$\frac{2}{3}$ &$\frac{1}{6}$,$\frac{1}{6}$,$\frac{5}{6}$,$\frac{5}{6}$ &$4X^3+4X$ \\ \hline
  \addtocounter{rownum-2}{1}\arabic{rownum-2} & $X^4+2X^3+3X^2+2X+1$ &$X^4-X^3+2X^2-X+1$  &$\frac{1}{3}$,$\frac{1}{3}$,$\frac{2}{3}$,$\frac{2}{3}$ &$\frac{1}{4}$,$\frac{3}{4}$,$\frac{1}{6}$,$\frac{5}{6}$ &$3X^3+X^2+3X$ \\ \hline
  
\end{tabular}
\label{table:arithmeticity not known}
Table \ref{table:arithmeticity not known} continued...
\end{table}}
 
 \pagebreak

{\tiny\begin{table}[h]
Table \ref{table:arithmeticity not known} continued...
\addtolength{\tabcolsep}{-2pt}
  \centering
\begin{tabular}{ |c|  c|   c| c| c| c|}
    \hline
    No. & $f(X)$ & $g(X)$ & $\alpha$ & $\beta$ & $f(X)-g(X)$  \\ \hline

  \addtocounter{rownum-2}{1}\arabic{rownum-2} & $X^4+2X^3+3X^2+2X+1$ &$X^4-X^3+X^2-X+1$  &$\frac{1}{3}$,$\frac{1}{3}$,$\frac{2}{3}$,$\frac{2}{3}$ &$\frac{1}{10}$,$\frac{3}{10}$,$\frac{7}{10}$,$\frac{9}{10}$ &$3X^3+2X^2+3X$ \\ \hline
  
  \addtocounter{rownum-2}{1}\arabic{rownum-2} & $X^4+3X^3+4X^2+3X+1$ &$X^4-2X^3+2X^2-2X+1$  &$\frac{1}{2}$,$\frac{1}{2}$,$\frac{1}{3}$,$\frac{2}{3}$ &0,0,$\frac{1}{4}$,$\frac{3}{4}$ &$5X^3+2X^2+5X$ \\ \hline
  
  \addtocounter{rownum-2}{1}\arabic{rownum-2} & $X^4+3X^3+4X^2+3X+1$ &$X^4+2X^2+1$  &$\frac{1}{2}$,$\frac{1}{2}$,$\frac{1}{3}$,$\frac{2}{3}$  &$\frac{1}{4}$,$\frac{1}{4}$,$\frac{3}{4}$,$\frac{3}{4}$ &$3X^3+2X^2+3X$ \\ \hline
  
  \addtocounter{rownum-2}{1}\arabic{rownum-2} & $X^4+3X^3+4X^2+3X+1$ &$X^4-3X^3+4X^2-3X+1$  &$\frac{1}{2}$,$\frac{1}{2}$,$\frac{1}{3}$,$\frac{2}{3}$  &0,0,$\frac{1}{6}$,$\frac{5}{6}$ &$6X^3+6X$ \\ \hline
  
  \addtocounter{rownum-2}{1}\arabic{rownum-2} & $X^4+3X^3+4X^2+3X+1$ &$X^4-2X^3+3X^2-2X+1$  &$\frac{1}{2}$,$\frac{1}{2}$,$\frac{1}{3}$,$\frac{2}{3}$  &$\frac{1}{6}$,$\frac{1}{6}$,$\frac{5}{6}$,$\frac{5}{6}$ &$5X^3+X^2+5X$ \\ \hline
  
  \addtocounter{rownum-2}{1}\arabic{rownum-2} & $X^4+3X^3+4X^2+3X+1$ &$X^4-X^3+2X^2-X+1$  &$\frac{1}{2}$,$\frac{1}{2}$,$\frac{1}{3}$,$\frac{2}{3}$  &$\frac{1}{4}$,$\frac{3}{4}$,$\frac{1}{6}$,$\frac{5}{6}$ &$4X^3+2X^2+4X$ \\ \hline
  
  \addtocounter{rownum-2}{1}\arabic{rownum-2} & $X^4+3X^3+4X^2+3X+1$ &$X^4+1$  &$\frac{1}{2}$,$\frac{1}{2}$,$\frac{1}{3}$,$\frac{2}{3}$  &$\frac{1}{8}$,$\frac{3}{8}$,$\frac{5}{8}$,$\frac{7}{8}$ &$3X^3+4X^2+3X$ \\ \hline
  
  \addtocounter{rownum-2}{1}\arabic{rownum-2} & $X^4+3X^3+4X^2+3X+1$ &$X^4-X^3+X^2-X+1$  &$\frac{1}{2}$,$\frac{1}{2}$,$\frac{1}{3}$,$\frac{2}{3}$  &$\frac{1}{10}$,$\frac{3}{10}$,$\frac{7}{10}$,$\frac{9}{10}$ &$4X^3+3X^2+4X$ \\ \hline
  
  \addtocounter{rownum-2}{1}\arabic{rownum-2} & $X^4+3X^3+4X^2+3X+1$ &$X^4-X^2+1$  &$\frac{1}{2}$,$\frac{1}{2}$,$\frac{1}{3}$,$\frac{2}{3}$  &$\frac{1}{12}$,$\frac{5}{12}$,$\frac{7}{12}$,$\frac{11}{12}$ &$3X^3+5X^2+3X$ \\ \hline
  
  \addtocounter{rownum-2}{1}\arabic{rownum-2} & $X^4-2X^3+2X^2-2X+1$ &$X^4+X^3+X^2+X+1$  &0,0,$\frac{1}{4}$,$\frac{3}{4}$ &$\frac{1}{5}$,$\frac{2}{5}$,$\frac{3}{5}$,$\frac{4}{5}$ &$-3X^3+X^2-3X$ \\ \hline
  
  \addtocounter{rownum-2}{1}\arabic{rownum-2} & $X^4-2X^3+2X^2-2X+1$ &$X^4+X^3+X+1$  &0,0,$\frac{1}{4}$,$\frac{3}{4}$ &$\frac{1}{2}$,$\frac{1}{2}$,$\frac{1}{6}$,$\frac{5}{6}$ &$-3X^3+2X^2-3X$ \\ \hline
  
  \addtocounter{rownum-2}{1}\arabic{rownum-2} & $X^4+2X^2+1$ &$X^4-3X^3+4X^2-3X+1$  &$\frac{1}{4}$,$\frac{1}{4}$,$\frac{3}{4}$,$\frac{3}{4}$  &0,0,$\frac{1}{6}$,$\frac{5}{6}$ &$3X^3-2X^2+3X$ \\ \hline
  
  \addtocounter{rownum-2}{1}\arabic{rownum-2} & $X^4+2X^3+2X^2+2X+1$ &$X^4-3X^3+4X^2-3X+1$  &$\frac{1}{2}$,$\frac{1}{2}$,$\frac{1}{4}$,$\frac{3}{4}$  &0,0,$\frac{1}{6}$,$\frac{5}{6}$ &$5X^3-2X^2+5X$ \\ \hline
  
  \addtocounter{rownum-2}{1}\arabic{rownum-2} & $X^4+2X^3+2X^2+2X+1$ &$X^4-2X^3+3X^2-2X+1$  &$\frac{1}{2}$,$\frac{1}{2}$,$\frac{1}{4}$,$\frac{3}{4}$  &$\frac{1}{6}$,$\frac{1}{6}$,$\frac{5}{6}$,$\frac{5}{6}$ &$4X^3-X^2+4X$ \\ \hline
  
  \addtocounter{rownum-2}{1}\arabic{rownum-2} & $X^4+2X^3+2X^2+2X+1$ &$X^4-X^3+X^2-X+1$  &$\frac{1}{2}$,$\frac{1}{2}$,$\frac{1}{4}$,$\frac{3}{4}$  &$\frac{1}{10}$,$\frac{3}{10}$,$\frac{7}{10}$,$\frac{9}{10}$ &$3X^3+X^2+3X$ \\ \hline
  
  \addtocounter{rownum-2}{1}\arabic{rownum-2} & $X^4+X^3+2X^2+X+1$ &$X^4-3X^3+4X^2-3X+1$  &$\frac{1}{3}$,$\frac{2}{3}$,$\frac{1}{4}$,$\frac{3}{4}$  &0,0,$\frac{1}{6}$,$\frac{5}{6}$ &$4X^3-2X^2+4X$ \\ \hline
  
  \refstepcounter{rownum-2}\arabic{rownum-2}\label{bigcoefficient} & $ X^4+X^3+2X^2+X+1$ &$ X^4-2X^3+3X^2-2X+1$  &${\frac 13}, {\frac 23},
{\frac 14}, {\frac 34}$ & $ {\frac 16},  {\frac 16}, 
{\frac 56}, {\frac 56}$ &$ 3X^3-X^2+3X$ \\ \hline
  
  \addtocounter{rownum-2}{1}\arabic{rownum-2} & $X^4+X^3+X^2+X+1$ &$X^4-3X^3+4X^2-3X+1$  &$\frac{1}{5}$,$\frac{2}{5}$,$\frac{3}{5}$,$\frac{4}{5}$  &0,0,$\frac{1}{6}$,$\frac{5}{6}$ &$4X^3-3X^2+4X$ \\ \hline
  
  \addtocounter{rownum-2}{1}\arabic{rownum-2} & $X^4+X^3+X^2+X+1$ &$X^4-2X^3+3X^2-2X+1$  &$\frac{1}{5}$,$\frac{2}{5}$,$\frac{3}{5}$,$\frac{4}{5}$  &$\frac{1}{6}$,$\frac{1}{6}$,$\frac{5}{6}$,$\frac{5}{6}$ &$3X^3-2X^2+3X$ \\ \hline
  
  \addtocounter{rownum-2}{1}\arabic{rownum-2} & $X^4-3X^3+4X^2-3X+1$ &$X^4+1$  &0,0,$\frac{1}{6}$,$\frac{5}{6}$  &$\frac{1}{8}$,$\frac{3}{8}$,$\frac{5}{8}$,$\frac{7}{8}$ &$-3X^3+4X^2-3X$ \\ \hline
  
  \addtocounter{rownum-2}{1}\arabic{rownum-2} & $X^4-3X^3+4X^2-3X+1$ &$X^4-X^2+1$  &0,0,$\frac{1}{6}$,$\frac{5}{6}$  &$\frac{1}{12}$,$\frac{5}{12}$,$\frac{7}{12}$,$\frac{11}{12}$ &$-3X^3+5X^2-3X$ \\ \hline

\end{tabular}
\end{table}}

\pagebreak
{\tiny\begin{table}[h]
\addtolength{\tabcolsep}{-2pt}
\caption{Monodromy associated to the Calabi-Yau threefolds. (In the following list,  $\alpha=(0,0,0,0)$ i.e. $f(X)=(X-1)^4=X^4-4X^3+6X^2-4X+1$)}

\setcounter{rownum-2}{0}
\centering
\begin{tabular}{ |c|  c|   c| c| c| c|}
\hline

  No. & $g(X)$ & $\beta$ & $f(X)-g(X)$ & Arithmetic\\ \hline
  
  {\bf \refstepcounter{rownum-2}\arabic{rownum-2}\label{arithmeticYYCE-1}} &${\bf X^4-2X^3+3X^2-2X+1} $ & $ {\bf 
\frac{1}{6}}$,${\bf \frac{1}{6}}$,${\bf \frac{5}{6}}$,${\bf \frac{5}{6} }$ & ${\bf -2X^3+3X^2-2X}$ & {\bf Yes} \\ \hline

  {\bf \refstepcounter{rownum-2}\arabic{rownum-2}*\label{brav1}} &${\bf X^4+4X^3+6X^2+4X+1}$ &${\bf \frac{1}{2}}$,${\bf \frac{1}{2}}$,${\bf \frac{1}{2}}$,${\bf \frac{1}{2}}$ &${\bf -8X^3-8X}$ &{\bf No}  \\ \hline

 \addtocounter{rownum-2}{1}\arabic{rownum-2} & $X^4+2X^3+3X^2+2X+1$& $\frac{1}{3}$,$\frac{1}{3}$,$\frac{2}{3}$,$\frac{2}{3}$ &$-6X^3+3X^2-6X$ & ? \\ \hline
  
{\bf   \refstepcounter{rownum-2}\arabic{rownum-2}*\label{brav2}}  &${\bf X^4+3X^3+4X^2+3X+1}$ &${\bf \frac{1}{2}}$,${\bf \frac{1}{2}}$,${\bf \frac{1}{3}}$,${\bf \frac{2}{3}}$ &${\bf -7X^3+2X^2-7X}$ &{\bf No}	 \\ \hline
  
  \addtocounter{rownum-2}{1}\arabic{rownum-2}  &$X^4+2X^2+1$  &$\frac{1}{4}$,$\frac{1}{4}$,$\frac{3}{4}$,$\frac{3}{4}$ &$-4X^3+4X^2-4X$ & ?\\ \hline
  
{\bf   \refstepcounter{rownum-2}\arabic{rownum-2}*\label{brav3}} &${\bf X^4+2X^3+2X^2+2X+1}$ &${\bf \frac{1}{2}}$,${\bf \frac{1}{2}}$,${\bf \frac{1}{4}}$,${\bf \frac{3}{4}}$ &${\bf -6X^3+4X^2-6X}$&{\bf No} \\ \hline
  
\addtocounter{rownum-2}{1}\arabic{rownum-2}  &$X^4+X^3+2X^2+X+1$ &$\frac{1}{3}$,$\frac{2}{3}$,$\frac{1}{4}$,$\frac{3}{4}$ &$-5X^3+4X^2-5X$ &? \\ \hline
  
  {\bf\refstepcounter{rownum-2}\arabic{rownum-2}*\label{brav4}}  & ${\bf X^4+X^3+X^2+X+1}$ &${\bf 
\frac{1}{5}}$,${\bf \frac{2}{5}}$,${\bf \frac{3}{5}}$,${\bf \frac{4}{5}}$ &${\bf -5X^3+5X^2-5X}$ &{\bf No}\\ \hline
  
{\bf   \refstepcounter{rownum-2}\arabic{rownum-2}*\label{brav5}}  &${\bf X^4+X^3+X+1}$ &${\bf \frac{1}{2}}$,${\bf \frac{1}{2}}$,${\bf \frac{1}{6}}$,${\bf \frac{5}{6}}$ &${\bf -5X^3+6X^2-5X}$&{\bf No} \\ \hline
  
  \addtocounter{rownum-2}{1}\arabic{rownum-2}  &$X^4+X^2+1$ &$\frac{1}{3}$,$\frac{2}{3}$,$\frac{1}{6}$,$\frac{5}{6}$ &$-4X^3+5X^2-4X$ &?\\ \hline
  
  {\bf\refstepcounter{rownum-2}\arabic{rownum-2}\label{arithmeticYYCE-2}}  &${\bf X^4-X^3+2X^2-X+1}$ &${\bf \frac{1}{4}}$,${\bf \frac{3}{4}}$,${\bf \frac{1}{6}}$,${\bf \frac{5}{6}}$ &${\bf -3X^3+4X^2-3X}$ &{\bf Yes} \\ \hline
  
  {\bf \refstepcounter{rownum-2}\arabic{rownum-2}*\label{brav6}} &${\bf X^4+1}$ &${\bf \frac{1}{8}}$,${\bf \frac{3}{8}}$,${\bf 
\frac{5}{8}}$,${\bf \frac{7}{8}}$ &${\bf -4X^3+6X^2-4X}$&{\bf No}\\ \hline
  
  {\bf\refstepcounter{rownum-2}\arabic{rownum-2}\label{arithmeticYYCE-3}} &${\bf X^4-X^3+X^2-X+1}$ &${\bf \frac{1}{10}}$,${\bf \frac{3}{10}}$,${\bf \frac{7}{10}}$,${\bf \frac{9}{10}}$ &${\bf -3X^3+5X^2-3X}$ &{\bf Yes} \\ \hline
  
  {\bf \refstepcounter{rownum-2}\arabic{rownum-2}*\label{brav7}} &${\bf X^4-X^2+1}$ &${\bf \frac{1}{12}}$,${\bf \frac{5}{12}}$,${\bf \frac{7}{12}}$,${\bf \frac{11}{12}}$ &${\bf -4X^3+7X^2-4X}$ &{\bf No} \\ \hline
  \end{tabular}
\label{table:calabiyau}
\end{table}}

\pagebreak

\end{document}